\documentclass{amsart}

\usepackage{amsmath, amsthm,amssymb, amsfonts}
\usepackage{rotating}
\usepackage{url}
\usepackage{microtype}
\usepackage{tikz-cd,extarrows}
\usepackage{cite}
\usepackage{comment}
\usepackage{enumitem,kantlipsum}

\renewcommand{\phi}{\varphi}

\newcommand{\est}[1]{\begin{equation*}\begin{split}#1\end{split}\end{equation*}}

\renewcommand{\P}{\mathbb{P}}
\newcommand{\cC}{\mathcal{C}}
\newcommand{\legen}[2]{\left(\frac{#1}{#2}\right)}

\newcommand{\Mod}[1]{\ (\mathrm{mod}\ #1)}

\def\Z{{\mathbb Z}}

\def\Fr{{\operatorname{Frob}}}
\def\supp{{\operatorname{supp}}}

\def\F{{\mathbb F}}
\def\Proj{{\mathbb P}}

\def\wt{{\operatorname{wt}}}

\def\Aut{{\operatorname{Aut \;}}}

\def\ev{{\operatorname{ev}}}
\def\Conf{{\operatorname{Conf}}}

\newtheorem{question}{Question}
\newtheorem{thm}{Theorem}

\newtheorem{prop}[thm]{Proposition}
\newtheorem{cor}[thm]{Corollary}
\newtheorem{lemm}[thm]{Lemma}
\newtheorem*{defn}{Definition}

\newtheorem{remark}[thm]{Remark}

\begin{document}

\title[Intersection Points of Plane Cubic Curves]{Counting Plane Cubic Curves over Finite Fields with a Prescribed Number of Rational Intersection Points}
\author[Kaplan]{Nathan Kaplan}
\address{Mathematics Department\\University of California Irvine\\Irvine, CA 92697}
\email{nckaplan@math.uci.edu}

\author[Matei]{Vlad Matei}
\address{Raymond and Beverly School of Mathematical Sciences\\Tel Aviv University \\Tel Aviv, Israel  69978}
\email{vladmatei@tau.ac.il}

\maketitle

\begin{abstract}
For each integer $k \in [0,9]$, we count the number of plane cubic curves defined over a finite field $\F_q$ that do not share a common component and intersect in exactly $k\ \F_q$-rational points.  We set this up as a problem about a weight enumerator of a certain projective Reed-Muller code.  The main inputs to the proof include counting pairs of cubic curves that do share a common component, counting configurations of points that fail to impose independent conditions on cubics, and a variation of the MacWilliams theorem from coding theory.
\end{abstract}

\section{Introduction}\label{sec:intro}

The goal of this paper is to answer the following question:
\begin{question}\label{Q1}
B\'ezout's theorem implies that a pair of plane cubic curves that do not share a common component intersect in at most $9$ points.  Let $k \in [0,9]$ be an integer.  \textbf{How many pairs of plane cubic curves defined over $\F_q$ do not share a common component and intersect in exactly $k\ \F_q$-rational points?}
\end{question}
\noindent\textbf{Note}: In this paper we use the phrase \emph{common component} to mean common component defined over $\overline{\F_q}$.

Let $\F_q[x_0,\ldots, x_n]_d$ denote the $\binom{n+d}{d}$-dimensional vector space of homogeneous degree $d$ polynomials in $\F_q[x_0,\ldots, x_n]$.  It does not make sense to evaluate a polynomial $f \in \F_q[x_0,\ldots, x_n]_d$ at a point of $\Proj^n(\F_q)$, since for any $\alpha \in \F_q^*$,
\[
f(\alpha x_0, \alpha x_1,\ldots, \alpha x_n) = \alpha^d f(x_0,x_1,\ldots, x_n).
\]
However, it does make sense to ask whether $f$ is zero at a point of $\Proj^n(\F_q)$ or not.  Therefore, we can rephrase Question \ref{Q1} as follows.
\begin{question}\label{Q2}
\begin{enumerate}
\item For each $k \in [0,9]$, how many of the $q^{20}$ ordered pairs $(f,g)$ with $f,g \in \F_q[x,y,z]_3$ have exactly $k$ common zeroes in $\Proj^2(\F_q)$?

\item More precisely, how many such pairs $(f,g)$ do not have a common irreducible factor over $\overline{\F_q}$ and have exactly $k$ common zeroes in $\Proj^2(\F_q)$?
\end{enumerate}
\end{question}
\noindent This is the main question that we answer in this paper.

\begin{thm}\label{ThmIntCubics}
Let $\F_q$ be a finite field of size $q > 2$. For each $k \in [0,9]$, let $c_k$ denote the number of pairs $(f,g)$ with $f,g \in \F_q[x,y,z]_3$, both nonzero, that do not have a common irreducible factor over $\overline{\F_q}$ and have exactly $k$ common zeros in $\Proj^2(\F_q)$.  We have
\begin{eqnarray*}
c_0 & = &  	
\frac{16687}{45360} \cdot (q + 1)^{2} \cdot (q - 1)^{3}
\cdot q^{5} \cdot (q^{2} + q + 1) \cdot \bigg(q^{8} - q^{7} +
\frac{15988}{16687} q^{6} \\
& & + \frac{882}{16687} q^{5} - \frac{126}{451}
q^{4} + \frac{3192}{16687} q^{3} + \frac{4397}{16687} q^{2} -
\frac{2507}{16687} q - \frac{2170}{16687}\bigg) \\
c_1 & = & \frac{2119}{5760} \cdot (q + 1) \cdot (q - 1)^{2} \cdot
q^{3} \cdot (q^{2} + q + 1) \cdot \\
& & \bigg(q^{12} - \frac{1}{14833} q^{11} +
\frac{2390}{2119} q^{10} - \frac{10240}{14833} q^{9} +
\frac{2459}{2119} q^{8} + \frac{99}{2119} q^{7} + \frac{3440}{2119}
q^{6} \\
& & - \frac{8630}{14833} q^{5} - \frac{4748}{2119} q^{4} +
\frac{76978}{14833} q^{3} + \frac{100}{2119} q^{2} - \frac{14160}{2119}
q + \frac{5760}{2119}\bigg) \\
c_2 & = & \frac{103}{560} \cdot (q - 1)^{2} \cdot (q + 1)^{2} \cdot
q^{4} \cdot (q^{2} + q + 1) \cdot \bigg(q^{10} + \frac{1}{927} q^{9} +
\frac{1634}{927} q^{8} + \frac{742}{927} q^{7} \\
& &  + \frac{1589}{927} q^{6}
+ \frac{1729}{927} q^{5} + \frac{4106}{927} q^{4} - \frac{2818}{103}
q^{3} + \frac{21608}{309} q^{2} - \frac{22610}{309} q +
\frac{2520}{103}\bigg) \\
c_3 & = &  \frac{53}{864} \cdot (q + 1)^{2} \cdot (q - 1)^{3} \cdot
q^{4} \cdot (q^{2} + q + 1) \cdot \bigg(q^{9} + \frac{527}{265} q^{8} +
\frac{221}{53} q^{7} + \frac{1533}{265} q^{6}\\
& &  + \frac{738}{53} q^{5} +
\frac{5958}{265} q^{4} - \frac{3956}{53} q^{3} + \frac{67402}{265} q^{2}
- \frac{11348}{53} q + \frac{2376}{53}\bigg) \\
c_4 & = & \frac{11}{720} \cdot (q + 1)^{2} \cdot (q - 1)^{3} \cdot
q^{4} \cdot (q^{2} + q + 1) \cdot \bigg(q^{9} + \frac{34}{11} q^{8} +
\frac{48}{11} q^{7} + \frac{182}{11} q^{6} \\
& & + \frac{109}{11} q^{5} -
\frac{1564}{11} q^{4} + 712 q^{3} - \frac{13292}{11} q^{2} +
\frac{7120}{11} q - \frac{600}{11}\bigg) 
\end{eqnarray*}
\begin{eqnarray*}
c_5 & = & \frac{1}{320} \cdot (q + 1)^{2} \cdot (q - 1)^{4} \cdot
q^{4} \cdot (q^{2} + q + 1) \cdot \bigg(q^{8} + \frac{40}{9} q^{7} +
\frac{151}{9} q^{6} + \frac{50}{9} q^{5} \\
 & & - \frac{874}{9} q^{4} +
\frac{2890}{3} q^{3} - \frac{7022}{3} q^{2} + \frac{4940}{3} q - 40\bigg) \\
c_6 & = & \frac{1}{2160} \cdot (q + 1)^{2} \cdot (q - 1)^{4} \cdot
q^{5} \cdot (q^{2} + q + 1) \cdot \\
& & \bigg(q^{7} + 9 q^{6} - 5 q^{5} - 17 q^{4}  + 910 q^{3} - 4316 q^{2} + 7416 q - 4670\bigg) \\
c_7 & = & \frac{1}{10080} \cdot (q - 2) \cdot (q + 1)^{2} \cdot (q -
1)^{4} \cdot q^{5} \cdot (q^{2} + q + 1)\cdot \\
& & \bigg(q^{6} + 2 q^{5} + 25
q^{4} + 288 q^{3} - 1692 q^{2} + 3574 q - 3290\bigg) \\
c_8 & = & \frac{1}{5040} \cdot (q - 3) \cdot (q - 2) \cdot (q +
1)^{2} \cdot (q - 1)^{4} \cdot q^{5} \cdot (q^{2} + q + 1)\cdot \\
& &  \bigg(q^{4}
+ 6 q^{3} - 31 q^{2} + 69 q - 105\bigg) \\
c_9 & = & \frac{1}{362880} \cdot (q - 2) \cdot (q + 1)^{2} \cdot (q
- 1)^{4} \cdot q^{5} \cdot (q^{2} + q + 1) \cdot \\
& & \bigg(q^{6} + 2 q^{5} - 73
q^{4} + 344 q^{3} - 838 q^{2} + 1754 q - 2030\bigg).
\end{eqnarray*}
\end{thm}

\begin{remark}\label{Rem2}
\begin{enumerate}[wide, labelwidth=!, labelindent=0pt]  
\item Each $c_k$ is a polynomial in $q$ of degree $20$, except $c_8$ which is a polynomial degree $19$.  The $q^{20}$ coefficient of each $c_k$ is the proportion of elements in the symmetric group $S_9$ that have exactly $k$ fixed points.  These `main terms' follow from recent work of Entin \cite{Entin}, which we discuss in more detail below.

\item One could consider the general problem, counting the number of pairs $(f,g)$ with $f\in \F_q[x,y,z]_d$ and $g\in \F_q[x,y,z]_e$ that have a given number of common zeros in $\Proj^2(\F_q)$.  On the way to proving Theorem \ref{ThmIntCubics} we give analogous, but far simpler, answers in the cases $(d,e) = (2,2)$ and $(d,e) = (3,2)$.  

In forthcoming work, we investigate the case where $e=2$ and $d$ is arbitrary, and obtain polynomial formulas.  The case of two cubics seems to be a kind of boundary.  We do not expect any case where $d,e \ge 3$ and $(d,e) \neq (3,3)$ for which the number of polynomials with exactly $de$ common zeros in $\Proj^2(\F_q)$ is a polynomial in $q$.  We discuss this further in Section \ref{sec:further}.

\end{enumerate}
\end{remark}

\subsection{Relationship to Previous Work}

The main terms that occur in the statement of Theorem \ref{ThmIntCubics}, and in the analogous results for $(d,e) \in \{(2,2), (3,2)\}$, follow from recent work of Entin \cite{Entin}.  Two projective plane curves defined over $\F_q$ that intersect transversely, one of degree $d$ and one of a degree $e$, give rise to a permutation in $S_{de}$ corresponding to the action of Frobenius, $\Fr_q$, on the $de$ intersection points of the two curves. The following result is a version of \cite[Corollary 1.3]{Entin}.

\begin{thm}[Entin]\label{ThmEntin}
Let $d$ and $e$ be positive integers and let $k \in [0,de]$ be an integer. Let $c_k(d,e)$ denote the number of pairs $(f,g)$ with $f \in \F_q[x,y,z]_d$ and $g \in \F_q[x,y,z]_e$ that have exactly $k$ common zeros in $\Proj^2(\F_q)$.  Then
\[
c_k(d,e) = \frac{\pi(k,de)}{(de)!} q^{\binom{d+2}{2}+\binom{e+2}{2}}\left(1+\text{O}_{de}(q^{-1/2})\right),
\]
where $\pi(k,de)$ is the number of permutations in $S_{de}$ with exactly $k$ fixed points.
\end{thm}

The approach taken in \cite{Entin} is to show that the monodromy group relevant to this problem is the symmetric group $S_{de}$.  The Frobenius action gives rise to all possible permutations in this group, and  applying the Chebotarev density theorem gives this quantitative result.  Because of the use of Chebotarev density, this approach cannot give precise quantitative statements like the one in Theorem \ref{ThmIntCubics}, and cannot distinguish between polynomial formulas and non-polynomial formulas for these kinds of counting problems.  

\begin{remark}
The statement of Theorem \ref{ThmEntin} does not include any assumptions about transversality since \cite[Corollary 1.3]{Entin} also implies that the number of pairs  of polynomials defining curves with non-transversal intersection can be absorbed into the error term.
\end{remark}

Theorem \ref{ThmIntCubics} fits into a body of literature on error-correcting codes that come from families of genus $1$ curves.  For an introduction to these ideas, see the surveys of Hurt \cite{Hurt}, and Schoof \cite{SchoofSurvey}.  There is an extensive discussion of classical algebraic geometry codes arising from elliptic curves in \cite[Section 4.4.2]{TVN}.  In Section \ref{sec:background} we review results about the projective Reed-Muller code whose codewords come from plane cubic curves.  Using results of Deuring, Waterhouse, and Schoof \cite{Deuring,Schoof, Waterhouse}, Elkies computes the weight enumerator of this code in \cite{Elkies}.  Kaplan studies the weight enumerator of the dual of this code, and the Reed-Muller code that comes from affine plane cubic curves in \cite{KaplanAGCT}.  Van der Geer, Schoof, and Van der Vlugt \cite{vdGSvdV}, and Schoof and Van der Vlugt \cite{SvdV}, study Zetterberg and Melas codes whose codewords come from families of elliptic curves in characteristics $2$ and $3$.  Finally, Kaplan and Petrow study the quadratic residue weight enumerator of a certain Reed-Solomon code in \cite{KaplanPetrowTAMS}.  Computing this weight enumerator involves counting isomorphism classes of elliptic curves with a fixed number of $\F_q$-points and with some additional structure related to the $2$-torsion of $E(\F_q)$.

There are many cases in which results from algebraic geometry and number theory are used to study families of error-correcting codes.  For just one reference, see the book \cite{TVN}.  It is much less common to find examples where techniques from coding theory are used to prove new results in arithmetic geometry.  This is the perspective we adopt in this paper, as one of the main inputs into the proof of Theorem \ref{ThmIntCubics} is a generalization of a classical theorem of MacWilliams about weight enumerators of linear codes and their duals.  The main idea for this project is inspired by work of Elkies \cite{Elkies}.

\subsection{Outline of the Paper}
In the next section we recall some coding-theoretic background and rephrase Question \ref{Q2} as a problem about the second weight enumerator of the projective Reed-Muller code whose codewords correspond to plane cubic curves.  We recall some previous results about this code and its dual.   In Section \ref{sec:dual_coef}, we explain how low-weight coefficients of weight enumerators of duals of Reed-Muller codes are related to interpolation problems in algebraic geometry.  We carefully analyze small collections of points that fail to impose independent conditions on conics and cubics.  In Section \ref{sec:twoconics}, we prove analogues of Theorem \ref{ThmIntCubics} for intersections of plane conics, and in Section \ref{sec:con_cub}, we prove the analogue of Theorem \ref{ThmIntCubics} for the intersection of a conic and a cubic.  In Section \ref{sec:large_int}, we count pairs of cubic curves that share a common component.  In Section \ref{sec:pf_thm1}, we complete the proof of Theorem \ref{ThmIntCubics}.  Finally, in Section \ref{sec:further} we discuss questions for further study.  
\begin{footnote}{
There is a large computational component to the results of this paper. All computations were done in the computer algebra system Sage.  We have made the programs used for this project available at the website of the first author: \\
\url{https://www.math.uci.edu/~nckaplan/research_files/intersections_of_cubics_sage_worksheets/}.}
\end{footnote}

In Section \ref{sec:dual_coef}, we use a theorem of Kaplan to compute the number of collections of $9$ points in $\Proj^2(\F_q)$ such that there are two cubics intersecting at these points that do not share a common component.  The result that we cite contains the additional restriction that the characteristic of $\F_q$ is not $2$ or $3$.  In Appendix \ref{AppendixA}, we compute the relevant quantity in a different way that works for any $\F_q$, which shows that the first part of \cite[Theorem 3]{KaplanAGCT} does hold in characteristics $2$ and $3$.

\section{Weight Enumerators of Reed-Muller Codes and their Duals}\label{sec:background}

In this section we recall some coding theory background and express Question \ref{Q2} as a problem about the second weight enumerator of a projective Reed-Muller code.
\begin{defn}
A subset $C \subseteq \F_q^n$ is called a \emph{code} of length $n$ over $\F_q$.  It is a \emph{linear code} if it is a linear subspace of $\F_q^n$.  For $x= (x_1, x_2, \ldots, x_n)$ and $y= (y_1, y_2,\ldots, y_n)$ in $\F_q^n$, the \emph{Hamming distance} between $x$ and $y$ is defined by
\[
d(x,y) = \#\{i\in [1,n] \mid x_i \neq y_i\}.
\]
The \emph{Hamming weight} of $x \in \F_q^n$, denoted $\wt(x)$, is
\[
\wt(x) = \#\{i\in [1,n] \mid x_i \neq 0\}.
\]
The \emph{Hamming weight enumerator} of $C$ is defined by
\begin{eqnarray*}
W_C(X,Y) & = & \sum_{c \in C} X^{n-\wt(c)} Y^{\wt(c)}  =  \sum_{i=0}^n A_i X^{n-i} Y^i,
\end{eqnarray*}
where 
\[
A_i = \#\{c\in C \mid \wt(c) = i\}.
\]
\end{defn}

The main problems we study are not about zeros of individual polynomials, but about common zeros of pairs of polynomials.  We make extensive use of the following variation of the Hamming weight enumerator.
\begin{defn}
Let $C \subseteq \F_q^n$ be a code. The \emph{second weight enumerator} of $C$ is 
\begin{eqnarray*}
W^{[2]}_C(X,Y) & = &  \sum_{c_1, c_2 \in C} X^{n-\wt(c_1, c_2)} Y^{\wt(c_1, c_2)} =  \sum_{i=0}^n A^{[2]}_i X^{n-i} Y^i,
\end{eqnarray*}
where $\wt(c_1, c_2)$ is the number of coordinates in which $c_1$ and $c_2$ are not both $0$.  
\end{defn}

The \emph{support} of an element $x \in\F_q^n$, denoted $\supp(x)$, is the set of coordinates in which $x$ is nonzero.  Therefore, $\wt(x) = |\supp(x)|$.  If $S \subseteq \F_q^n$ is a linear subspace, its support, denoted $\supp(S)$, is the set of coordinates in which at least one element of $S$ is nonzero. The weight of $S$ is $\wt(S) = |\supp(S)|$.  For $x_1,\ldots, x_k \in \F_q^n$, let $\wt(x_1,\ldots, x_k)$ be the weight of the subspace spanned by $x_1,\ldots, x_k$.  This is consistent with the definition of $\wt(c_1, c_2)$ given above.  We also write $\supp(x_1,\ldots, x_k)$ for the union of the supports of $x_1,\ldots, x_k$, which is the same as the support of the subspace spanned by $x_1,\ldots, x_k$.

There is a variation of the second weight enumerator for two codes that are not necessarily the same.
\begin{defn}
Let $C_1,C_2 \subseteq \F_q^n$. The second weight enumerator of $C_1, C_2$ is 
\begin{eqnarray*}
W^{[2]}_{C_1, C_2}(X,Y) & = & \sum_{\substack{c_1 \in C_1\\ c_2 \in C_2}} X^{n-\wt(c_1, c_2)} Y^{\wt(c_1, c_2)}  =  \sum_{i=0}^n A^{[2]}_i X^{n-i} Y^i.
\end{eqnarray*}
\noindent In the case $C_1 = C_2$, we write $W_C^{[2]}(X,Y)$ instead of $W_{C,C}^{[2]}(X,Y)$.
\end{defn}

It is often useful to divide up the second weight enumerator of a code $C \subseteq \F_q^n$ by the dimension of the subspace of $\F_q^n$ spanned by the pair $c_1, c_2$.  There are $q^2-1$ ordered pairs of vectors that span a chosen one-dimensional subspace of $\F_q^n$ and $(q^2-1)(q^2-q)$ choices for an ordered pair that span a chosen two-dimensional subspace.  Therefore,
\begin{equation}\label{EqSupp}
W^{[2]}_C(X,Y) = X^n + (q^2-1) W^{(1)}_C(X,Y) + (q^2-1)(q^2-q) W^{(2)}_C(X,Y),
\end{equation}
where 
\[
W^{(r)}_C(X,Y) = \sum_{S \subseteq C} X^{n-\wt(S)} Y^{\wt(S)}, 
\]
and the sum is taken over all $r$-dimensional subspaces of $C$.  We see that 
\[
W_C(X,Y) = X^n + (q-1) W^{(1)}_C(X,Y).
\]

\subsection{Reed-Muller codes}

We introduce the main class of codes that we study in this paper.  For ease of notation, let $N = |\Proj^n(\F_q)| = (q^{n+1}-1)/(q-1)$.  Let $p_1, p_2, \ldots, p_N$ be an ordering of the elements of $\Proj^n(\F_q)$, and let $p_1', p_2',\ldots, p_N'$ be a choice of affine representatives for these points.

We define the evaluation map:
\begin{eqnarray*}
\operatorname{ev} \colon \F_q[x_0,x_1,\ldots, x_n]_d & \mapsto & \F_q^N \\
f & \mapsto & \left(f(p'_1),\ldots, f(p'_N)\right).
\end{eqnarray*}
It is not difficult to see that $\ev(\alpha f + g) = \alpha\cdot \ev(f) + \ev(g)$, so the image of this map is a linear subspace of $\F_q^N$.  As long as there does not exist a degree $d$ polynomial vanishing at every point in $\Proj^n(\F_q)$, which is the case for $q \ge d$, this map is injective, so the image of this map has dimension $\binom{n+d}{d}$.  For the rest of the paper, we write $C_{n,d}$ for $\ev(\F_q[x_0,x_1,\ldots, x_n]_d)$, and refer to it as a \emph{projective Reed-Muller code}.

\begin{remark}
\begin{enumerate}[wide, labelwidth=!, labelindent=0pt]  
\item This definition depends on an ordering of the points of $\Proj^n(\F_q)$ and on a choice of affine representatives for these projective points.  These Reed-Muller codes satisfy a strong form of equivalence.  In particular, the weight enumerators that we study in this paper do not depend on these choices.  See \cite[Section 1.7]{HP} or \cite[Remark 1]{KaplanAGCT}.

\item As in \cite{KaplanAGCT}, the definition of projective Reed-Muller code given here is not the same as the one given by Lachaud \cite{Lachaud}, but it equivalent to it if one makes a standard choice of affine representatives.

\item Throughout this paper we focus on the code $C_{2,3}$, which is $10$-dimensional when $q>2$.  Whenever we refer to $C_{2,3}$ we will assume that $q>2$, even if we do not explicitly state this assumption.

\end{enumerate}
\end{remark}

The classical, or affine, Reed-Muller code is defined as follows.  Let $\F_q[x_1,\ldots, x_n]_{\le d}$ denote the $\binom{n+d}{d}$-dimensional vector space of polynomials in $\F_q[x_1,\ldots, x_n]$ with degree at most $d$.  Let $p_1, p_2, \ldots, p_{q^n}$ be an ordering of the elements of $\F_q^n$. We define the evaluation map:
\begin{eqnarray*}
\operatorname{ev} \colon \F_q[x_1,\ldots, x_n]_{\le d} & \mapsto & \F_q^{q^n} \\
f & \mapsto & \left(f(p_1),\ldots, f(p_{q^n})\right).
\end{eqnarray*}
The image of this map is a linear subspace of $\F_q^{q^n}$, and the evaluation map is injective when $q-1 \ge d$.  We write $C^{A}_{n,d}$ for $\ev(\F_q[x_1,\ldots, x_n]_{\le d})$, and refer to it as an \emph{affine Reed-Muller code}.

\begin{remark}
\begin{enumerate}[wide, labelwidth=!, labelindent=0pt]  
\item The affine Reed-Muller code $C_{n,d}^A$ is monomially equivalent to $C_{n,d}$ punctured at the set of $(q^n-1)/(q-1)$ coordinates corresponding to a choice of hyperplane at infinity.  See \cite[Section 1]{KaplanAGCT} for details.

\item Coordinates of codewords of Reed-Muller codes and their duals correspond to points.  Throughout this paper, we will use the term `support', to refer both to subsets of coordinates of codewords, and also to the underlying collections of points.
\end{enumerate}
\end{remark}

We now can state the main problem of this paper in this language of weight enumerators of codes.
\begin{question}
Let 
\[
W^{[2]}_{C_{2,3}}(X,Y) = \sum_{i=0}^{q^2+q+1} A^{[2]}_i X^{q^2+q+1-i} Y^i.
\]
What are the coefficients $A^{[2]}_{q^2+q+1-9}, A^{[2]}_{q^2+q+1-8}, \ldots, A^{[2]}_{q^2+q+1}$?
\end{question}
\noindent The answer to this question is almost equivalent to the statement of Theorem \ref{ThmIntCubics}.  One difference is that there is a contribution to $W^{[2]}_{C_{2,3}}(X,Y)$ from pairs of codewords $(c_1, c_2)$ such that $c_1, c_2$ span a linear subspace of $\F_q^{q^2+q+1}$ of dimension at most $1$.  Another difference is that there is a contribution from pairs $(c_1, c_2)$ for which the corresponding cubic curves are not equal, but share a common component.

\subsection{The dual code of a linear code and the MacWilliams theorem}

\begin{defn}
For $x= (x_1, x_2, \ldots, x_n)$ and $y= (y_1, y_2,\ldots, y_n)$ in $\F_q^n$, let
\[
\langle x,y \rangle = \sum_{i=1}^n x_i y_i \in \F_q.
\]
Let $C \subseteq \F_q^n$ be a linear code. The \emph{dual code} of $C$, denoted $C^{\perp}$, is 
\[
C^\perp = \{y \in \F_q^n \mid \langle x,y \rangle = 0\ \forall x\in \F_q^n\}.
\]
\end{defn}

A theorem of MacWilliams says that the weight enumerator of a linear code determines the weight enumerator of its dual.
\begin{thm}[MacWilliams]\label{MacWilliamsThm}
Let $C \subseteq \F_q^n$ be a linear code.  Then
\[
W_{C^\perp}(X,Y) = \frac{1}{|C|} W_C(X+(q-1)Y, X-Y).
\]
\end{thm}

There is an extensive literature about variations of Theorem \ref{MacWilliamsThm}.  See for example \cite[Chapter 5]{MacWilliamsSloane}.  We will make extensive use of the following MacWilliams theorem for the second weight enumerator of $C_1, C_2 \subseteq \F_q^n$.  See for example, \cite[Theorem 2]{Shiromoto} with $a=X$ and $b=c=d = Y$.
\begin{thm}\label{MacWilliamsThm2}
Let $C_1, C_2 \subseteq \F_q^n$ be linear codes.  We have
\[
W^{[2]}_{C^\perp_1, C^\perp_2}(X,Y)  = \frac{1}{|C_1| \cdot |C_2|}W^{[2]}_{C_1, C_2}(X+(q^2-1)Y, X-Y). 
\]
\end{thm}

\subsection{Weight enumerators of Reed-Muller codes}  

Since each nonzero linear form on $\Proj^n$ defines a hyperplane, it is easy to see that
\[
W_{C_{n,1}}(X,Y) = X^{\frac{q^{n+1}-1}{q-1}} + (q^{n+1}-1) X^{\frac{q^{n}-1}{q-1}} Y^{q^n}.
\]
It is also straightforward to compute $W_{C_{1,d}}(X,Y)$ for any $d$, for example by noting that $C_{1,d}$ is an MDS code \cite[Chapter 11]{MacWilliamsSloane}.  Elkies computes $W_{C_{n,2}}(X,Y)$ and $W_{C^\perp_{n,2}}(X,Y)$ in \cite{Elkies}.  Aubry gives additional results about weight enumerators of Reed-Muller codes associated to quadric hypersurfaces in \cite{Aubry}.  Elkies also computes the weight enumerator of the code from plane cubics, $W_{C_{2,3}}(X,Y)$, and the weight enumerator of the code from cubic surfaces, $W_{C_{3,3}}(X,Y)$ \cite{Elkies}.  These are the only cases for which $W_{C_{n,d}}(X,Y)$ is known exactly.  Many authors have studied minimum distances and other invariants of affine and projective Reed-Muller codes.  For example, see \cite{DGM, vanderGeerSchoofvanderVlugtC, HeijnenPellikaan, Lachaud, Sorensen}. 

We recall expressions for $W_{C_{2,2}}(X,Y)$ and $W_{C_{2,3}}(X,Y)$ since we will need them later.  A nonzero $f \in \F_q[x,y,z]_2$ defines either a double line, a pair of $\F_q$-rational lines, a pair of Galois-conjugate lines defined over $\F_{q^2}$ but not over $\F_q$, or a smooth conic.  A double line and a smooth conic each have $q+1\ \F_q$-rational points.  A pair of rational lines has $2q+1$ rational points.  A pair of Galois-conjugate lines has $1$ rational point.  Since there are $q^2+q+1\ \F_q$-rational lines and $(q^2+q+1)(q^2-q)/2$ pairs of Galois-conjugate lines, it is straightforward to prove the following result.
\begin{prop}\label{WC22prop}
We have 
\begin{eqnarray*}
W_{C_{2,2}}(X,Y) & = & X^{q^2+q+1} + \frac{(q^2+q+1)q(q-1)^2}{2} X Y^{q^2+q} \\
& & + (q^3-q^2+1)(q^2+q+1)(q-1) X^{q+1} Y^{q^2} \\
& & +  \frac{(q^2+q+1)(q+1)q(q-1)}{2} X^{2q+1} Y^{q^2-q}.
\end{eqnarray*}
\end{prop}
\begin{remark}
The irreducible cubic factor in the $X^{q+1} Y^{q^2}$ coefficient of this weight enumerator comes from the fact that it is equal to the sum of two terms that factor nicely.  This is a common phenomenon that we will see throughout this paper.
\end{remark}

We also recall $W_{C_{2,2}^A}(X,Y)$.  One can think of an affine conic as a projective conic minus the points of the line at infinity $L$.  For example, a smooth affine conic has either $q-1, q$, or $q+1\ \F_q$-rational points depending on $\#(C\cap L)(\F_q)$, where $C$ is the corresponding projective conic.  A straightforward calculation gives the following~result.
\begin{prop}
Let $q>2$.  We have 
\begin{eqnarray*}
W_{C^A_{2,2}}(X,Y) & = & X^{q^2} + \frac{(q-1)(q^3-q+2)}{2} Y^{q^2} + \frac{(q-1)^2 q^3}{2} X Y^{q^2-1} \\
& & + \frac{(q-1)^2 q^3 (q+1)}{2} X^{q-1} Y^{q^2-q+1} + (q^3-q)(q^2-q+2) X^{q} Y^{q^2-q} \\
& & + \frac{(q-1)^3 q^3}{2} X^{q+1} Y^{q^2-q-1} +  \frac{(q-1)(q+1)q^3}{2} X^{2q-1} Y^{q^2-2q+1} \\
& & +  \frac{q(q+1)(q-1)^2}{2} X^{2q} Y^{q^2-2q}.
\end{eqnarray*}
\end{prop}

We recall the necessary background to state the formula for $W_{C_{2,3}}(X,Y)$ given by Elkies \cite{Elkies}; see also \cite{KaplanAGCT}.  We write
\[
W_{C_{2,3}}(X,Y) = W^{\text{sing}}_{C_{2,3}}(X,Y)  + W^{\text{smooth}}_{C_{2,3}}(X,Y),
\]
where $W^{\text{sing}}_{C_{2,3}}(X,Y)$ is the contribution to $W_{C_{2,3}}(X,Y)$ from cubic polynomials, including the zero polynomial, that define singular cubic curves, and $W^{\text{smooth}}_{C_{2,3}}(X,Y)$ is the contribution to $W_{C_{2,3}}(X,Y)$ from cubic polynomials that define smooth cubic curves.  An expression for $W^{\text{sing}}_{C_{2,3}}(X,Y)$ is given as \cite[Lemma 2]{KaplanAGCT}.

\begin{lemm}
Let $q \ge 3$.  We have 
\begin{eqnarray*}
& & W^{\text{sing}}_{C_{2,3}}(X,Y)  =  X^{q^2+q+1} + 
\frac{(q^3-1)(q^3-q)}{6} X^{3q+1} Y^{q^2-2q} \\
& + &
\frac{(q^3-1) (q^4+q^3)}{6}  X^{3q} Y^{q^2-2q+1} + 
\frac{(q^3-1)(q^3-q^2)(q^2-q)}{2} X^{2q+2} Y^{q^2-q-1} \\
& + &
(q^3-1)(q^2+q)(q^2-q+1) X^{2q+1} Y^{q^2-q} 
 +  
\frac{(q^6-q^3)(q^2-1) }{2} X^{2q} Y^{q^2-q+1} \\
& + &
\frac{(q^3-1)(q^6-q^5)}{2}  X^{q+2} Y^{q^2-1} 
 + 
\frac{(q^3-1)(2q^5-q^3-q+2)}{2} X^{q+1} Y^{q^2} \\
  & + &
\frac{(q^3-1)(q^3-q)(q^3-q^2)}{2} X^q Y^{q^2+1} 
 +
\frac{(q^3-1)(q^3-q)}{3} X Y^{q^2+q} \\
 & +  &
\frac{(q-1)(q^3-q) (q^3-q^2)}{3} Y^{q^2+q+1}.
\end{eqnarray*}
\end{lemm}

In order to describe the contribution from $W_{C_{2,3}}(X,Y)$ from smooth cubic curves, we need to count the number of smooth cubic curves with a given number of $\F_q$-points.  This leads to counting elliptic curves defined over $\F_q$ with a given number of $\F_q$-points.   We closely follow the presentation in \cite{KaplanAGCT}.  

When we mention an elliptic curve $E$ defined over $\F_q$ we always implicitly mean the isomorphism class of $E$.  With this convention, let $\cC=\{E/\F_q\}$ denote the set of $\F_q$-isomorphism classes of elliptic curves defined over $\F_q$.  For a description of this set along with $\#\Aut_{\!\F_q\!}(E)$ for each curve, see \cite[Proposition 5.7]{Schoof}.  These counts imply
\[
\sum_{E\in \cC} \frac{1}{\#\Aut_{\!\F_q\!}(E)} = q,
\]
so the finite set $\cC$ is a probability space where a singleton $\{E\}$ occurs with~probability
\[
\P_q(\{E\}) = \frac{1}{q \# \Aut_{\!\F_q\!}(E)}.
\]
Let $t_E\in \Z$ denote the trace of the Frobenius endomorphism associated to $E$.  We have $t_E=q+1 -\#E(\F_q)$ and by Hasse's theorem $t^2_E \leq 4q$.  For an integer $t$, let $\cC(t)$ be the subset of $\cC$ for which $t_E = t$.  The following is \cite[Proposition 1]{KaplanAGCT}.

\begin{prop}\label{WE_C23_Smooth}
Let $q \ge 3$.  Then 
\[
W_{C_{2,3}}^{\text{smooth}}(X,Y) = (q^3-1)(q^3-q)(q^3-q^2) q \sum_{t^2 \le 4q} \P_q(\cC(t)) X^{q+1-t} Y^{q^2+t}.
\]
\end{prop}

We now need only give an expression for $ \P_q(\cC(t))$.  We recall some terminology related to class numbers of imaginary quadratic fields.  
\begin{defn}
For $d<0$ with $d \equiv 0,1\Mod{4}$, let $h(d)$ be the class number of the unique imaginary quadratic order of discriminant $d$.  Let 
\est{
h_w(d) = 
\begin{cases} h(d)/3, &\text{ if } d = -3, \\ 
h(d)/2, & \text{ if } d=-4, \\ 
h(d) & \text{ if } d < 0,\ d \equiv 0,1 \Mod{4}, \text{ and } d\neq -3,-4, \\
0 & \text{otherwise}, 
\end{cases}
} 
and for $\Delta \equiv 0,1\Mod 4$ let 
\est{
\label{HKcn}H(\Delta) = \sum_{d^2 \mid \Delta} h_w\left(\frac{\Delta}{d^2}\right)
} 
be the \emph{Hurwitz-Kronecker class number}. For $a\in \Z$ and $n$ a positive integer, the \emph{Kronecker symbol} $\Big(\frac{a}{n}\Big)$ is defined to be the completely multiplicative function in $n$ such that if $p$ is an odd prime $\legen{a}{p}$ is the quadratic residue symbol and if $p=2$ 
\est{
\left(\frac{a}{2}\right) = \begin{cases} 0 & \text{ if } 2 \mid a, \\ 
1 & \text{ if } a \equiv \pm 1 \Mod 8, \\ -1 & \text{ if } a \equiv \pm5 \Mod 8 .
\end{cases} 
}
\end{defn}

The following is a weighted version of \cite[Theorem 4.6]{Schoof}, which builds on earlier work of Deuring and Waterhouse \cite{Deuring, Waterhouse}.  
\begin{lemm}\cite[Lemma 2]{KaplanPetrow}\label{S46withweights}
Let $t\in \Z$.  Suppose $q = p^v$ where $p$ is prime and $v\geq 1$.  Then if $q$ is not a square
\begin{alignat*}{3}
\P_q(\cC(t))  = &  \frac{1}{2q}H(t^2-4q) \quad && \text{ if } t^2 < 4q \text{ and } p\nmid t,\\
 = & \frac{1}{2q}H(-4p)\quad  && \text{ if } t=0, \\
  = & \frac{1}{4q} \quad  && \text{ if } t^2=2q \text{ and } p =2, \\
 = & \frac{1}{6q} && \text{ if } t^2=3q \text{ and } p =3, 
\end{alignat*}
 and if $q$ is a square
\begin{alignat*}{3}
\P_q(\cC(t))  = &  \frac{1}{2q}H(t^2-4q)\quad  && \text{ if } t^2 < 4q \text{ and } p\nmid t,\\
 = & \frac{1}{4q}\left(1 - \legen{-4}{p}\right) \quad   && \text{ if } t=0, \\
 = & \frac{1}{6q}\left(1 - \legen{-3}{p}\right) \quad  && \text{ if } t^2 = q, \\
  = & \frac{p-1}{24q}\quad   && \text{ if } t^2 = 4q,
\end{alignat*}
 and $\P_q(\cC(t)) = 0$ in all other cases.
\end{lemm}
Combining Proposition \ref{WE_C23_Smooth} and Lemma \ref{S46withweights} gives an expression for $W_{C_{2,3}}^{\text{smooth}}(X,Y)$ in terms of class numbers of orders in imaginary quadratic fields.

We move on to the computation of the weight enumerators for the dual codes. We apply Theorem \ref{MacWilliamsThm} to $W_{C_{2,2}}(X,Y)$ to determine the low-weight coefficients of~$W_{C_{2,2}^\perp}(X,Y)$.
\begin{prop}\label{WC22Dualprop}
Let 
\[
W_{C_{2,2}^\perp}(X,Y) = \sum_{i=0}^{q^2+q+1} B_i X^{q^2+q+1 - i } Y^i.
\]
Then $B_0 = 1,\ B_1 = B_2 = B_3 = 0$, and 
\begin{eqnarray*}
B_4 & = & (q-1)(q^2+q+1) \binom{q+1}{4}\\
B_5 & = & \left((q^2-1) - \binom{5}{4}(q-1)\right)(q^2+q+1) \binom{q+1}{5}\\
B_6 & = & \left((q^3 - 1) - \binom{6}{5} (q^2-1)+\binom{6}{4} (q-1)\right) (q^2+q+1) \binom{q+1}{6} \\
& & + (q-1) (q^5-q^2) \binom{q+1}{6} + (q-1) \binom{q^2+q+1}{2} \binom{q}{3}^2.
\end{eqnarray*}
\end{prop}

Applying Theorem \ref{MacWilliamsThm} to $W_{C^A_{2,2}}(X,Y)$ gives the following result.
\begin{prop}
Suppose $q > 2$.  Let 
\[
W_{(C^{A}_{2,2})^\perp}(X,Y) = \sum_{i=0}^{q^2} B_i X^{q^2 - i } Y^i.
\]
Then $B_0 = 1,\ B_1 = B_2 = B_3 = 0$, and $B_4 = (q-1)(q^2+q)\binom{q}{4}$.
\end{prop}

Applying Theorem \ref{MacWilliamsThm} to our expression for $W_{C_{2,3}}(X,Y)$ determines the low-weight coefficients of $W_{C_{2,3}^\perp}(X,Y)$.  The following result is given in a different form as the first part of \cite[Theorem 3]{KaplanAGCT}.  
\begin{thm}\label{C23LowWeightDual}
Let $\F_q$ be a finite field of size $q>2$, and
\[
W_{C_{2,3}^\perp}(X,Y) = \sum_{i=0}^{q^2+q+1} B_i X^{q^2+q+1 - i } Y^i.
\]
Then $B_0 = 1,\ B_1 = B_2 = B_3 = B_4 = 0$, and 
\begin{eqnarray*}
B_5 & = & (q-1)(q^2+q+1) \binom{q+1}{5}\\
B_6 & = & \left((q^2-1) - \binom{6}{5}(q-1)\right)(q^2+q+1) \binom{q+1}{6}\\
B_7 & = & \left((q^3 - 1) - \binom{7}{6} (q^2-1)+\binom{7}{5} (q-1)\right) (q^2+q+1) \binom{q+1}{7} \\
B_8 & = & \left((q^4 - 1) - \binom{8}{7} (q^3-1)+\binom{8}{6} (q^2-1) - \binom{8}{5}(q-1)\right) (q^2+q+1) \binom{q+1}{8}  \\
& & + (q-1) (q^5-q^2) \binom{q+1}{8} + (q-1) \binom{q^2+q+1}{2} \binom{q}{4}^2 \\
B_9 & = & (q-1)I_9(q) +(q^5-q^2)\dbinom{q+1}{9}\left((q^2-1)-9(q-1)\right)\\  
& & +(q^2+q+1)\dbinom{q+1}{9}(q^5 - 9q^4 + 36q^3 - 84q^2 + 126q - 70) \\
& & +2\dbinom{q^2+q+1}{2}\dbinom{q}{5}\dbinom{q}{4}(q^2-6q+5)+\dbinom{q^2+q+1}{2}\dbinom{q}{4}^2(q^2-3q+2),
\end{eqnarray*}
where 
\[
I_9(q) =\frac{1}{9!}(q^6 + 2q^5 - 73q^4 + 344q^3 - 838q^2 + 1754q - 2030)(q^2 + q + 1)(q + 1)(q - 1)^2(q - 2)q^4.
\]
\end{thm}

\noindent We will explain the form of these dual code coefficients in the next section.

\begin{remark}
In \cite[Theorem 3]{KaplanAGCT} there is an additional restriction that the characteristic of $\F_q$ is not $2$ or $3$.  It seems that the proof of this result does work without any changes for all $q>2$.  In the next section, we give a different proof of the formulas for $B_i$ when $i \le 8$ and also for $B_9$, except for the computation of $I_9(q)$.  In Appendix \ref{AppendixA} we compute $I_9(q)$ in a different way that is independent of the characteristic of $\F_q$, which completes a separate proof of the formula for $B_9$.
\end{remark}

\section{Dual Code Coefficients and Configurations of Points that Fail to Impose Independent Conditions}\label{sec:dual_coef}

\subsection{Statement of Results}
The goal of this section is to analyze the supports of low-weight codewords of $C_{2,2}^\perp$ and $C_{2,3}^\perp$ and prove the following three results.

The first result will be used in Section \ref{sec:twoconics} to prove an analogue of Theorem \ref{ThmIntCubics} for intersections of plane conics.
\begin{lemm}\label{lemmalowweightconics}
Let
\[
W^{[2]}_{C_{2,2}^\perp}(X,Y) = \sum_{i=0}^{q^2+q+1} B^{[2]}_i X^{q^2+q+1 - i } Y^i.
\]
We have $B^{[2]}_0 = 1,\ B^{[2]}_1 = B^{[2]}_2 = B^{[2]}_3 = 0$, and $B^{[2]}_4 = (q^2-1)(q^2+q+1)\binom{q+1}{4}$.
\end{lemm}

The next result will be used in Section \ref{sec:con_cub} to prove an analogue of Theorem \ref{ThmIntCubics} for the intersection of a conic and a cubic.
\begin{lemm}\label{lemmalowweightconiccubic}
Suppose $q>2$.  Let
\begin{eqnarray*}
W_{C_{2,2}^\perp}(X,Y) & = & \sum_{i=0}^{q^2+q+1} B^{2,2}_i X^{q^2+q+1 - i } Y^i, \\
W_{C_{2,3}^\perp}(X,Y) & = & \sum_{i=0}^{q^2+q+1} B^{2,3}_i X^{q^2+q+1 - i } Y^i,\\
W^{[2]}_{C_{2,2}^\perp,C_{2,3}^\perp}(X,Y) & = &   \sum_{i=0}^{q^2+q+1} B^{[2]}_i X^{q^2+q+1 - i } Y^i.
\end{eqnarray*}
Then $B^{[2]}_0 = 1,\ B^{[2]}_1 = B^{[2]}_2 = B^{[2]}_3 = 0$, and 
\begin{eqnarray*}
B^{[2]}_4 & = & B^{2,2}_4\\
B^{[2]}_5 & = & B^{2,2}_5 + B^{2,3}_5 + (q^3-1) (q^2-1) \binom{q+1}{5} \\
B^{[2]}_6 & = &B^{2,2}_6 + B^{2,3}_6 +  (q^3-1)(q^2-1) (q^2+q-5) \binom{q+1}{6}.
\end{eqnarray*}
\end{lemm}

The final result plays a key role in the proof of Theorem \ref{ThmIntCubics}.  Let 
\begin{eqnarray*}
g_3(5) & = & (q-1)^2\\
g_3(6) & = & (q^2+2q-5) (q-1)^2 \\
g_3(7) & = & (q^4-2q^3 - 4 q^2 - 12 q + 15)(q-1)^2 \\
g_3(8) & = & (q^5+3q^4 - 2 q^3- 14q^2 - 7q + 35)(q-1)^3 \\
g_3(9) & = & (q^8+2q^7 - 6 q^6 - 14 q^5 + 14 q^4 + 40 q^3 - 112 q + 70) (q-1)^2.
\end{eqnarray*}
In Proposition \ref{gn_prop} we will give a general formula for $g_3(m)$.
\begin{lemm}\label{lemmalowweightcubics}
Suppose $q>2$.  Let
\begin{eqnarray*}
W^{[2]}_{C_{2,3}^\perp}(X,Y)  & = &  \sum_{i=0}^{q^2+q+1} B^{[2]}_i X^{q^2+q+1 - i } Y^i \\
W_{C_{2,3}^\perp}(X,Y) &  = &  \sum_{i=0}^{q^2+q+1} B^{2,3}_i X^{q^2+q+1 - i } Y^i.
\end{eqnarray*}
Then $B^{[2]}_0= 1,\ B^{[2]}_1 = B^{[2]}_2 = B^{[2]}_3 = B^{[2]}_4=0$, and 
\begin{eqnarray*}
B^{[2]}_5 & = & 2 B^{2,3}_5 + (q^2+q+1)\dbinom{q+1}{5}g_3(5)\\
B^{[2]}_6 & = & 2 B^{2,3}_6 + (q^2+q+1)\dbinom{q+1}{6}g_3(6)\\
B^{[2]}_7 & = & 2 B^{2,3}_7 + (q^2+q+1)\dbinom{q+1}{7}g_3(7)\\
B^{[2]}_8 & = & 2 B^{2,3}_8 + (q^2+q+1)\dbinom{q+1}{8}g_3(8) \\
& & +(q^5-q^2)\dbinom{q+1}{8} (q-1)^2+\dbinom{q^2+q+1}{2} \dbinom{q}{4}^2 (q-1)^2\\
B^{[2]}_9 & = & 2 B^{2,3}_9 + (q^2+q+1)\dbinom{q+1}{9}g_3(9)+ 2\dbinom{q^2+q+1}{2}\dbinom{q}{4}\dbinom{q}{5}(q^4-8q^2+12q-5)\\
& & +  \dbinom{q^2+q+1}{2}\dbinom{q}{4}^2(q^4-5q^2+6q-2)\\
& & +(q^5-q^2)\dbinom{q+1}{9}(q^4-11q^2+18q-8)+I_9(q) (q-1)^2,
\end{eqnarray*}
where $I_9(q)$ is given in Theorem \ref{C23LowWeightDual}.
\end{lemm}

\subsection{Supports of dual codewords and points failing to impose independent conditions}\label{sec:supports}

For any subset $S\subseteq \mathbb{P}^n$ one can consider the subspace of $\F_q[x_0,\ldots, x_n]_d$ defined by 
\[
\Gamma_{S,d}=\left\{f\in \F_q[x_0,\ldots, x_n]_d \mid f(p)=0\ \forall p\in S\right\}.
\] 
It is a well-known fact that $\Gamma_{S,d}$ has codimension at most $\# S$. If equality holds for this codimension then we say that $S$ \emph{imposes independent conditions} on degree $d$ forms, or equivalently, on degree $d$ hypersurfaces, in $\mathbb{P}^n$; otherwise we say that $S$ \emph{imposes dependent conditions} or \emph{fails to impose independent conditions}. If $S$ fails to impose independent conditions, there exist $b_p \in \F_q$, not all zero, such that $\sum_{p\in S} b_p f(p)=0$ for all $f \in \F_q[x_0,\ldots, x_n]_d$.  This produces a nonzero codeword in $C_{n,d}^\perp$ supported on the coordinates corresponding to $S$.  We conclude that supports of nonzero codewords in $C_{n,d}^{\perp}$ come from sets of points in $\mathbb{P}^n$ that fail to impose independent conditions on degree $d$ hypersurfaces.

The following classical result, which can be found in \cite{EGH}, is at the core of our analysis of low-weight codewords of $C_{2,d}^\perp$.  It gives a precise description for small sets of points that fail to impose independent condition on degree $d$ plane curves.

\begin{lemm}\label{EGHlemma}
Let $\Omega=\{p_1,\ldots, p_n\}\subset \mathbb{P}^2$ be any collection of $n\leq 2d+2$ distinct points. The points of $\Omega$ fail to impose independent conditions on curves of degree $d$ if only if either $d+2$ points of the points of $\Omega$ are collinear or $n=2d+2$ and $\Omega$ is contained in a conic.

\end{lemm}

\begin{remark} This lemma implies that there are no nonzero codewords of weight less then $d+2$ in $C_{2,d}^{\perp}$.

\end{remark}

Using Lemma \ref{EGHlemma} we deduce the following geometric characterizations for the support of low-weight codewords of $C_{2,2}^\perp$ and $C_{2,3}^\perp$.

\begin{lemm}\label{supportconiclemma}
Let $c \in C_{2,2}^{\perp}$ be nonzero with $\wt(c) = k$. Then $k\geq 4$ and moreover:
\begin{itemize}
\item If $k \in \{4, 5\}$ the support of $c$ consists of $k$ collinear points.
\item If $k=6$ then the support of $c$ is either:
\begin{itemize}
\item $6$ collinear points, 
\item $6$ points on a smooth conic, or 
\item $6$ points on two lines, with $3$ on each, and not including the intersection point of the two lines.
\end{itemize}
\end{itemize}

\end{lemm}

\begin{proof} 
Lemma \ref{EGHlemma} implies that a nonzero codeword of $C_{2,2}^\perp$ has weight at least $4$. When $k=4$ the $4$ points of the support must be collinear.

\vspace{.25cm}

\noindent{\textbf{k=5}}:  Let $\{P_1,\ldots, P_5\}$ be the points of $\Proj^2(\F_q)$ in the support of $c$.  Lemma \ref{EGHlemma} implies that at least $4$ of these $5$ points are collinear.  Without loss of generality suppose that $P_1,P_2,P_3,P_4$ are collinear.  There is a $c' \in C_{2,2}^\perp$ with $\wt(c') = 4$ supported on $P_1, P_2, P_3,P_4$.  There exists an $\alpha \in \F_q^*$ such that the coordinate corresponding to $P_1$ in $c - \alpha c'$ is equal to $0$.  Therefore, $c - \alpha c'$ is a nonzero codeword of weight at most $4$ supported on $P_2,P_3,P_4,P_5$.  By the result for $k=4$, these points are collinear.

\vspace{.25cm}

\noindent{\textbf{k=6}}:  Let $\{P_1,\ldots, P_6\}$ be the support of $c$.  By Lemma \ref{EGHlemma} we need only show that if at least $4$ of these points are collinear, then they are all collinear.  

Without loss of generality, suppose that $P_1,P_2, P_3, P_4$ are collinear.  As in the case $k=5$, there is a $c' \in C_{2,2}^\perp$ with $\wt(c') = 4$ supported on $P_1, P_2, P_3,P_4$, and there is an $\alpha \in \F_q^*$ such that the coordinate corresponding to $P_1$ in $c - \alpha c'$ is equal to $0$.  We see that $c - \alpha c'$ is a nonzero codeword of weight at most $5$ whose support includes $P_5$ and $P_6$.  By the result for $k\le 5$, the points of the support are collinear. Therefore, all $6$ points are collinear.

\end{proof}

\begin{lemm} \label{supportcubiclemma}
Let $c \in C_{2,3}^{\perp}$ be nonzero with $\wt(c) = k$. Then $k\geq 5$ and moreover:
\begin{itemize}
\item If $k \in \{5,6,7\}$ the support of $c$ consists of $k$ collinear points.
\item If $k=8$ then the support of $c$ is either:
\begin{itemize}
\item $8$ collinear points, 
\item $8$ points on a smooth conic, or 
\item $8$ points on two lines, with $4$ on each, and not including the intersection point of the two lines.
\end{itemize}
\item If $k=9$ then the support of $c$ is either:
\begin{itemize}
\item $9$ collinear points, 
\item $9$ points on a smooth conic, 
\item $9$ points on two lines, with $5$ points on the first line and $4$ on the second line, and not including the intersection points of the two lines,
\item $9$ points on two lines consisting of the intersection point of the two lines together with $4$ additional points on each, or
\item $9$ points that are exactly the intersection points of two cubics , i.e., the two cubics share no common component.
\end{itemize}
\end{itemize}

\end{lemm}

\begin{remark}\label{Chaslesremark}
The fact that a collection of $9$ points that are the intersection points of two cubics fails to impose independent conditions on cubics is due to Chasles, although is colloquially known as the Cayley-Bacharach theorem \cite[Theorem CB3]{EGH}.  We give a different form of this result in Proposition \ref{Hart4.5} in Appendix \ref{AppendixA}.
\end{remark}

\begin{proof} 
Lemma \ref{EGHlemma} implies that a nonzero codeword of $C_{2,3}^\perp$ has weight at least $5$.  When $k=5$ the $5$ points of the support must be collinear.

\vspace{.25cm}

\noindent{\textbf{k=6}}:  Let $\{P_1,\ldots, P_6\}$ be the points of $\Proj^2(\F_q)$ in the support of $c$.  Lemma \ref{EGHlemma} implies that $5$ of these points are collinear.  Without loss of generality suppose that $P_1,\ldots, P_5$ are collinear.  There is a $c' \in C_{2,3}^\perp$ with $\wt(c') = 5$ supported on $\{P_1 \ldots P_5\}$.  There exists an $\alpha \in \F_q^*$ such that the coordinate corresponding to $P_1$ in $c - \alpha c'$ is equal to $0$.  Therefore, $c - \alpha c'$ is a nonzero codeword of weight at most $5$ supported on $\{P_2, \ldots, P_6\}$ whose support includes $P_6$.  By the result for $k=5$, these points are collinear.

\vspace{.25cm}

\noindent{\textbf{k=7}}:  Let $\{P_1,\ldots, P_7\}$ be the support of $c$.  Lemma \ref{EGHlemma} implies that $5$ of these points are collinear.  Without loss of generality suppose that $P_1,\ldots, P_5$ are collinear.  There is a $c' \in C_{2,3}^\perp$ with $\wt(c') = 5$ supported on $\{P_1 \ldots P_5\}$.  There exists an $\alpha \in \F_q^*$ such that the coordinate corresponding to $P_1$ in $c - \alpha c'$ is equal to $0$.  Therefore, $c - \alpha c'$ is a nonzero codeword of weight at most $6$ supported on $\{P_2, \ldots, P_7\}$ whose support includes $P_6$ and $P_7$.  By the result for $k \le 6$, these points are collinear.

\vspace{.25cm}

\noindent{\textbf{k=8}}:  Let $\{P_1,\ldots, P_8\}$ be the support of $c$.  By Lemma \ref{EGHlemma} we need only show that if at least $5$ of these points are collinear, then they are all collinear.

Without loss of generality, suppose that $P_1,\ldots, P_5$ are collinear.  As in the case $k=7$, there is a $c' \in C_{2,3}^\perp$ with $\wt(c') = 5$ supported on $\{P_1, \ldots, P_5\}$, and there is an $\alpha \in \F_q^*$ such that the coordinate corresponding to $P_1$ in $c - \alpha c'$ is equal to $0$.  We see that $c - \alpha c'$ is a nonzero codeword of weight at most $7$ supported on $\{P_2, \ldots, P_8\}$ whose support includes $P_6, P_7$, and $P_8$.  By the result for $k \le 7$, these points are collinear.  

\vspace{.25cm}

\noindent{\textbf{k=9}}:  Let $\{P_1,\ldots, P_9\}$ be the support of $c$.  We break this argument into two cases.  In the first case, suppose that all subsets of $8$ of these points impose independent conditions on cubics.  We will show that there must be distinct cubic curves that intersect exactly at these $9$ points.  In the second case, we suppose that there exists a subset of $8$ of these points that imposes dependent conditions on cubics.  We will show that the $9$ points must be in one of the other configurations given in the statement.

\vspace{.25cm}

\noindent{\textbf{Case I}}:  Take any subset of $8$ points out of the $9$ points, say $\{P_1,\ldots, P_8\}$. They impose independent conditions on cubics, so there is a $2$-dimensional vector space of cubic polynomials that vanish at these points.  Let $f_1$ and $f_2$ be a basis for this vector space.  Since $\{P_1,\ldots, P_9\}$ fails to impose independent conditions on cubics, the dimension of the vector space of cubic polynomials vanishing at these points is at least $2$.   We conclude that this space has dimension $2$, and $\{f_1, f_2\}$ is a basis for it.  Let $C_1$ be the cubic defined by $f_1$ and $C_2$ be the cubic defined by $f_2$.  We see that $\{P_1,\ldots, P_9\} \subseteq C_1 \cap C_2$.  Two cubic curves that share a common a pair of Galois-conjugate lines defined over $\F_{q^2}$ but not over $\F_q$ intersect in at most $2\ \F_q$-points.  Therefore, if $C_1$ and $C_2$ share a common component, it must be either an $\F_q$-rational line or a (possibly reducible) conic.  If two distinct cubic curves intersect in at least $9\ \F_q$-points and share a unique $\F_q$-rational line, then that line must contain at least of the $5$ common $\F_q$-points.  If two distinct cubic curves intersect in at least $9\ \F_q$-points and share a common conic, then that conic must contain at least $8$ of the common $\F_q$-points.  Therefore, if $C_1$ and $C_2$ share a common component, it cannot be the case that every subset of $8$ of the $9$ points imposes independent conditions on cubics.

\vspace{.25cm}

\noindent{\textbf{Case II}}:  Suppose there is a subset of $8$ points from $\{P_1,\ldots, P_9\}$ that fails to impose independent conditions on cubics.  Without loss of generality, suppose this subset is $\{P_1,\ldots, P_8\}$.  By Lemma \ref{EGHlemma}, either all $8$ points lie on a conic, or at least $5$ of the points are collinear.

Suppose that $\{P_1,\ldots, P_8\}$ lie on a conic.  We first consider the case where that conic is smooth.  There is a codeword $c' \in C_{2,3}^\perp$ of weight $8$ supported on $\{P_1, \ldots, P_8\}$, and there is an $\alpha \in \F_q^*$ such that the coordinate corresponding to $P_1$ in $c - \alpha c'$ is equal to $0$.  Therefore, $c - \alpha c'$ is a nonzero codeword of weight at most $8$ supported on $\{P_2, \ldots, P_9\}$ with support including $P_9$.  It is clear that no $4$ points of $\{P_2,\ldots, P_9\}$ are collinear.  By the result for $k=8$, these points all lie on a smooth conic.

Suppose that $\{P_1,\ldots, P_8\}$ lie on a conic, but not on a smooth conic, and no $5$ of these points are collinear.  These $8$ points must lie on two lines with $4$ on each, and not including the intersection point.  Without loss of generality suppose $P_1,\ldots, P_4$ are collinear and $P_5,\ldots, P_8$ are collinear.  There is a codeword $c' \in C_{2,3}^\perp$ of weight $8$ supported on $\{P_1, \ldots, P_8\}$, and there is an $\alpha \in \F_q^*$ such that the coordinate corresponding to $P_1$ in $c - \alpha c'$ is $0$. Therefore, $c - \alpha c'$ is a codeword of weight at most $8$ supported on $\{P_2,\ldots, P_9\}$ with support including $P_9$.  By the result for $k \le 8$, either $5$ of these points are collinear, or these $8$ points lie on two lines with $4$ on each, and not including the intersection point. In this second case, $P_2,P_3,P_4$ and $P_9$ must be collinear, so $\{P_1,\ldots, P_9\}$ contains $5$ collinear points.

Now suppose that $5$ points of $\{P_1,\ldots, P_9\}$ are collinear.  We show that either all of the points are collinear, or the points lie on two lines with at least $4$ points on each.  Suppose $P_1,\ldots, P_5$ are collinear. There is a codeword $c' \in C_{2,3}^\perp$ of weight $5$ supported on $\{P_1, \ldots, P_5\}$, and there is an $\alpha \in \F_q^*$ such that the coordinate corresponding to $P_1$ in $c - \alpha c'$ is equal to $0$.  Therefore, $c - \alpha c'$ is a nonzero codeword of weight at most $8$ supported on $\{P_2, \ldots, P_9\}$ with support including $P_6, P_7, P_8$, and $P_9$.  Since $P_2,\ldots, P_5$ are collinear, the result for $k \le 8$ implies that $P_6,\ldots, P_9$ are collinear.  Either $P_2,\ldots, P_9$ are all collinear, which implies all $9$ points are collinear, or $P_2,\ldots, P_5$ and $P_6,\ldots, P_9$ lie on two different lines, not including the intersection point.  In this case, $P_1,\ldots, P_9$ lie on two lines with at least $4$ points on each.

%We next prove that if $5$ of the points are collinear, then all $9$ points lie on the union of two lines.  Suppose $P_1,\ldots, P_5$ are collinear.  There is a codeword $c' \in C_{2,3}^\perp$ of weight $5$ supported on $\{P_1, \ldots, P_5\}$, and there is an $\alpha \in \F_q^*$ such that the coordinate corresponding to $P_1$ in $c - \alpha c'$ is equal to $0$.  Therefore, $c - \alpha c'$ is a nonzero codeword of weight at most $8$ supported on $\{P_2, \ldots, P_9\}$ with support including $P_6, P_7, P_8$, and $P_9$.  Since $P_2,\ldots, P_5$ are collinear, the result for $k \le 8$ implies that $P_6,\ldots, P_9$ are collinear.  Therefore $P_2,\ldots, P_9$ lie on two lines and we conclude that $P_1,\ldots, P_9$ lie on two lines.

%We now prove that if $6$ of the points are collinear, then all $9$ are collinear.  Suppose $P_1,\ldots, P_6$ are collinear. There is a codeword $c' \in C_{2,3}^\perp$ of weight $5$ supported on $\{P_1, \ldots, P_5\}$, and there is an $\alpha \in \F_q^*$ such that the coordinate corresponding to $P_1$ in $c - \alpha c'$ is equal to $0$.  Therefore, $c - \alpha c'$ is a nonzero codeword of weight at most $8$ supported on $\{P_2, \ldots, P_9\}$ with support including $P_6, \ldots, P_9$.  Since $P_2,\ldots, P_6$ are collinear, the result for $k=8$ implies that $P_2,\ldots, P_9$ are collinear, and therefore, $P_1,\ldots, P_9$ are collinear.

\end{proof}

\subsection{Counting Codewords with Given Support}

As we see in the statements of Lemma \ref{supportconiclemma} and Lemma \ref{supportcubiclemma}, many of the low-weight codewords of $C_{2,2}^\perp$ and $C_{2,3}^\perp$ come from collections of collinear points. In this section we count the number of weight $k$ codewords supported on a set of collinear points.

\begin{prop}\label{Vdmdim}
Let $V_{d,m}$ the vector space of codewords in $C_{2,d}^\perp$ supported on a collection of $m \le q+1$ collinear points. Then $\dim(V_{d,m})=\max(m-d-1,0)$.
\end{prop}

\begin{proof} 
Let $P_1,\ldots, P_m$ be the collection of collinear points defining $V_{d,m}$.

The statement for $m\leq d+1$ is immediate from Lemma \ref{EGHlemma}, so we assume $m\geq d+2$. We prove this proposition by induction on $m$. 

Let $m=d+2$. By Lemma \ref{EGHlemma}, there is a nonzero $c_1 \in V_{d,m}$ supported on $P_1,\ldots, P_m$ and for every nonzero element of $V_{d,m}$ the coordinate corresponding each $P_i$ is nonzero. Suppose that $c_2 \in V_{d,m}$ where $c_1$ and $c_2$ are linearly independent.  Since $c_2$ must be nonzero, the coordinate corresponding to $P_1$ in $c_2$ is nonzero. There is an $\alpha \in \F_q$ such that the coordinate corresponding to $P_1$ in $c_1 - \alpha c_2$ is $0$. Since $c_1 - \alpha c_2$ is a codeword supported on a set of at most $d+1$ collinear points, Lemma \ref{EGHlemma} implies that it is the zero codeword.  This contradicts the assumption that $c_1, c_2$ were linearly independent.

For the induction step, suppose that $P_{m+1}$ is on the same line as $\{P_1,\ldots, P_m\}$.  Obviously $V_{d,m} \subseteq V_{d,m+1}$. Let $\phi \colon V_{d,m+1} \to \F_q$ be the map that takes an element of $V_{d,m+1}$ to the coordinate corresponding to $P_{m+1}$. It is clear that $\ker(\phi) = V_{d,m}$.  We need only show that $\phi$ is surjective, which is equivalent to saying that $\phi$ is nonzero.  By Lemma \ref{EGHlemma}, there is a nonzero $c \in C_{2,d}^\perp$ supported on $P_1,\ldots, P_{d+1}, P_{m+1}$, so $\wt(c) = d+2$.  This implies $\phi(c) \neq 0$.  We conclude that $V_{d,m+1}/V_{d,m}\cong \mathbb{F}_q$, and~thus 
\[
\dim(V_{d,m+1})=\dim(V_{d,m})+1=m-d-1+1=(m+1)-d-1.
\]
\end{proof}

\begin{prop}\label{fdmprop}
Let $d+2\le m \le q+1$.  The number of $c\in V_{d,m}$ with $\wt(c)=m$ is
\[
f_d(m)=\sum_{i=0}^{m-d-2} (q^{m-d-1-i}-1)\dbinom{m}{i}(-1)^{i}.
\]	
\end{prop} 

\begin{proof}
The proof follows from an inclusion-exclusion argument.  Let $P_1,\ldots, P_m$ be the collinear points defining $V_{d,m}$.  Let $A_i$ be the set of codewords for which the coordinate corresponding to $P_i$ is zero.  By Proposition \ref{Vdmdim},
\[
f_d(m)=q^{m-d-1}-\left|\bigcup_{i=1}^{m} A_i\right|.
\] 
We have
\[
\left|\bigcup_{i=1}^{m} A_i\right|=\sum_{\substack{S\subseteq \{1,2,\ldots m\} \\ S\neq \emptyset }} (-1)^{|S|-1} \left|\bigcap_{j\in S} A_j\right|.
\]
There is a natural identification
\[
 \bigcap_{j\in S} A_j \cong V_{d, m-|S|},
\]
so applying Proposition \ref{Vdmdim} completes the proof.

\end{proof}

To prove the results stated at the beginning of Section \ref{sec:dual_coef}, we need to count pairs of codewords $c_1, c_2$ such that $\wt(c_1, c_2)$ is fixed.
\begin{prop}\label{gn_prop}
Let $d+2\le m \le q+1$.  The number of (ordered) pairs $(c_1,c_2)$ with $c_1, c_2 \in V_{d,m}$ both nonzero with $\wt(c_1,c_2)=m$ is
\[
g_d(m)=\sum_{\substack{d+2\leq a,b\leq m\\ a+b\geq m}} \dbinom{m}{a}\dbinom{a}{b+a-m} f_d(a)f_d(b).
\]
\end{prop} 
\noindent This explains the values of $g_3(m)$ given before the statement of Lemma \ref{lemmalowweightcubics}.

\begin{proof} 
Let $P_1,\ldots, P_m$ be the collection of collinear points defining $V_{d,m}$.  Suppose $c_1,c_2 \in V_{d,m}$ satisfy $\wt(c_1,c_2) = m$.  Let $A \subseteq \{P_1,\ldots, P_m\}$ be the support of $c_1$, and $B \subseteq \{P_1,\ldots, P_m\}$ be the support of $c_2$.  By assumption $A\cup B = \{P_1,\ldots, P_m\}$.

Suppose $|A| = a$ and $|B| = b$, so $d+2 \le a,b \le m$ and $a+b \ge m$.  There are $\binom{m}{a}$ choices for $A$.  We see that $B$ must contain the $m-a$ points of $\{P_1,\ldots, P_m\} \setminus A$, and also $b-(m-a)$ points of $A$.  Applying Proposition \ref{fdmprop} completes the proof.

\end{proof}

We now turn from counting codewords supported on a set of collinear points, to counting codewords supported on a set of noncollinear points.

\begin{lemm}\label{VSconiclemma}
Let $S$ be a set of $6$ noncollinear points that fail to impose independent conditions on conics. Let $V_S$ be the vector space of codewords in $C_{2,2}^{\perp}$ supported on $S$.
\begin{itemize}[wide, labelwidth=!, labelindent=0pt] 
\item If $S$ consists of $6$ points on a smooth conic, then $V_S$ is $1$-dimensional and every nonzero element of $V_S$ has weight $6$.

\item If $S$ is a collection of $6$ points on two $\F_q$-rational lines, with $3$ on each line and not including the intersection point, then $V_S$ is $1$-dimensional and every nonzero element of $V_S$ has weight $6$.

\end{itemize}
\end{lemm}

\begin{proof} 
We prove the two statements together.  For each such set $S$, pick one of the points $P \in S$ and consider the natural map $\phi_P \colon V_S \rightarrow \F_q$ that takes a codeword to its value at the coordinate corresponding to $P$.

Let $c \in V_S$ be such that $\phi_P(c) = 0$.  Then $\wt(c) \le 5$, and Lemma \ref{supportconiclemma} implies that $c$ is the all zero codeword.  Therefore, $\phi_P$ is injective.  Lemma \ref{EGHlemma} implies that $S$ fails to impose independent conditions on conics, so there is a nonzero element of $V_S$.  This implies that $\phi_P$ is surjective.

\end{proof}

We now explain the form of the coefficients of $W_{C_{2,2}^\perp}(X,Y)$ given in Proposition~\ref{WC22prop}.
\begin{itemize}[wide, labelwidth=!, labelindent=0pt] 
\item Lemma \ref{supportconiclemma} and Proposition \ref{fdmprop} imply that every $c \in C_{2,2}^\perp$ with $\wt(c) = 4$ is supported on $4$ collinear points. Given any $4$ collinear points there are exactly $q-1$ weight $4$ codewords supported on them.

\item Similarly, every $c \in C_{2,2}^\perp$ with $\wt(c) = 5$ is supported on $5$ collinear points. Given any $5$ collinear points there are exactly $q^2-1 - \binom{5}{4}(q-1)$ weight $5$ codewords supported on them.

\item Lemma \ref{supportconiclemma} implies that every $c \in C_{2,2}^\perp$ with $\wt(c) = 6$ is supported either on a set of $6$ collinear points, on a set if $6$ points on a smooth conic, or on a set contained in two $\F_q$-rational lines, with $3$ on each line and not including the intersection point.  Applying Proposition \ref{fdmprop} and Lemma \ref{VSconiclemma} along with some elementary counting implies that the number of $c \in C_{2,2}^\perp$ with $\wt(c) = 6$ is
\[
(q^2+q+1) \binom{q+1}{6} f_2(6) + (q^5-q^2) \binom{q+1}{6}(q-1) + \binom{q^2+q+1}{2} \binom{q}{3}^2 (q-1).
\]
\end{itemize}

We now prove the analogue of Lemma \ref{VSconiclemma} for sets of noncollinear points that fail to impose independent conditions on cubics.  We consider codewords of weight $8$ and weight $9$ separately.

\begin{lemm}\label{VScubic8lemma}
Let $S$ be a set of $8$ noncollinear points that fail to impose independent conditions on cubics. Let $V_S$ be the vector space of codewords in $C_{2,3}^\perp$ supported on $S$.
\begin{itemize}[wide, labelwidth=!, labelindent=0pt] 
\item If $S$ consists of $8$ points on a smooth conic, then $V_S$ is $1$-dimensional and every nonzero element of $V_S$ has weight $8$.

\item If $S$ is a collection of $8$ points on two $\F_q$-rational lines, with $4$ on each line and not including the intersection point, then $V_S$ is $1$-dimensional and every nonzero element of $V_S$ has weight $8$.
\end{itemize}
\end{lemm}
\noindent The proof of this lemma is very similar to the proof of Lemma \ref{VSconiclemma}, so we omit it.

We use Lemma \ref{VScubic8lemma} to explain the form of the $B_8$ coefficient given in Theorem \ref{C23LowWeightDual}.  If $c \in C_{2,3}^\perp$ has $\wt(c) = 8$ then $S = \supp(c)$ is either a set of $8$ collinear points, or one of the two configurations given in the statement of Lemma \ref{VScubic8lemma}.  Proposition \ref{fdmprop} together with this lemma imply that the number of these weight $8$ codewords is 
\[
(q^2+q+1)\binom{q+1}{8} f_3(8) + (q^5-q^2)\binom{q+1}{8}(q-1) + \binom{q^2+q+1}{2} \binom{q}{4}^2 (q-1).
\]

We next give the analogous statement for codewords of $C_{2,3}^\perp$ of weight $9$.
\begin{lemm}\label{VScubic9lemma}
Let $S$ be a set of $9$ noncollinear points that fail to impose independent conditions on cubics. Let $V_S$ be the vector space of codewords in $C_{2,3}^{\perp}$ supported on $S$.
\begin{itemize}[wide, labelwidth=!, labelindent=0pt] 
\item If $S$ consists of $9$ points on a smooth conic, then $V_S$ is $2$-dimensional and every nonzero element of $V_S$ has weight $8$ or $9$.

\item If $S$ is a collection on $9$ points on two lines with at least $4$ points on each, then $V_S$ is $2$-dimensional and every nonzero element of $V_S$ has weight $5, 8,$ or $9$.

\item If $S$ is a collection of $9$ points that are exactly the intersection points of two cubics, i.e., the two cubics share no common component, then $V_S$ is $1$-dimensional  and every nonzero element of $V_S$ weight $9$.
\end{itemize}

\end{lemm}

\begin{proof} 
In each of the cases described above, we choose a subset $S' \subset S$ of size $8$.
\begin{itemize}[wide, labelwidth=!, labelindent=0pt] 
\item If $S$ is a collection of $9$ points on a smooth conic, then $S'$ is $8$ points on a smooth~conic.
\item If $S$ is a collection on $9$ points on two lines with at least $4$ points on each, then choose $S'$ to be a subset with $4$ points on each line, not including the intersection point.
\item In the final case in the statement of the lemma, choose $S' \subset S$ to be any subset of $8$ points. By Lemma \ref{supportcubiclemma}, $S'$ imposes independent conditions on cubics.
\end{itemize}
\noindent In each case, let $P$ denote the point of $S \setminus S'$.

Let $V_{S'}$ denote the vector space of codewords of $C_{2,3}^\perp$ supported on $S'$.  Clearly $V_{S'} \subseteq V_S$.  Consider the map $\phi_P \colon V_S \rightarrow \F_q$ that takes a codeword to the value of the coordinate corresponding to $P$.  It is clear that $\ker(\phi_P) = V_{S'}$.  Lemma \ref{VScubic8lemma} implies that in the first two cases, $\dim(V_{S'}) = 1$, and in the final case $V_{S'}$ consists of only the zero codeword, so $\phi_P$ is an isomorphism.

\end{proof}

We use Lemma \ref{VScubic9lemma} to explain the form of the $B_9$ coefficient given in Theorem \ref{C23LowWeightDual}.  If $c \in C_{2,3}^\perp$ has $\wt(c) = 9$ then $S = \supp(c)$ is either a set of $9$ collinear points, or one of the three configurations given in the statement of Lemma \ref{VScubic8lemma}.  
\begin{itemize}[wide, labelwidth=!, labelindent=0pt] 
\item Suppose $S$ consists of $9$ points on a smooth conic.  Lemma \ref{VScubic8lemma} implies that $V_S$ is $2$-dimensional.  We claim that there are $9(q-1)$ nonzero codewords of weight less than $9$ in $V_S$.  For any collection of $8$ points of $S$, there are $q-1$ codewords of weight $8$ supported on them.  Lemma \ref{supportcubiclemma} implies that these are the only nonzero codewords of weight less than $9$ in $V_S$.   

\item Suppose $S$ consists of $9$ points on two lines, with $5$ points on one and $4$ points on the other, not including the intersection point.  Lemma \ref{VScubic8lemma} implies that $V_S$ is $2$-dimensional.  We claim that there are $6(q-1)$ nonzero codewords of weight less than $9$ in $V_S$.

Let $S'$ be a subset of size $8$ that we get from removing one point from the line containing $5$ points of $S$. There are $q-1$ elements of $V_S$ of weight $8$ supported on $S'$.  There are also $q-1$ codewords of weight $5$ in $V_S$ supported on the $5$ collinear points of $S$.  Lemma \ref{supportcubiclemma} implies that these are the only nonzero codewords of weight less than $9$ in $V_S$.

\item Suppose $S$ consists of $9$ points on two lines, the intersection point together with $4$ additional points on each.   Lemma \ref{VScubic8lemma} implies that $V_S$ is $2$-dimensional.  We claim that there are $3(q-1)$ nonzero codewords of weight less than $9$ in $V_S$.

For each of the two subsets of $5$ collinear points in $S$ there are $q-1$ codewords of weight $5$ supported on these points.  There are also $q-1$ codewords of weight $8$ supported on the subset of $8$ points of $S$ not including the intersection point.  Lemma \ref{supportcubiclemma} implies that these are the only nonzero codewords of weight less than $9$ in $V_S$.
\end{itemize}

Proposition \ref{fdmprop} together with the observations above imply that the number of $c \in C_{2,3}^\perp$ with $\wt(c) = 9$ is
\begin{eqnarray*}
& &  (q^2+q+1)\dbinom{q+1}{9} f_3(9)+ (q^5-q^2)\binom{q+1}{9}\left(q^2-1-9(q-1)\right) \\
& + &  2\binom{q^2+q+1}{2}\binom{q}{5}\binom{q}{4}(q^2-1 - 6(q-1))\\
& + & \binom{q^2+q+1}{2}\binom{q}{4}^2(q^2-1-3(q-1)) + (q-1) I_9(q),
\end{eqnarray*}
where $I_9(q)$ is the number of collections of $9$ points in $\Proj^2(\F_q)$ that are the intersection points of two cubics that do not share a common component.  Comparing this expression to the one given in Theorem \ref{C23LowWeightDual} proves the following.

\begin{cor}\label{I9cor}
We have that $I_9(q)$ equals
\begin{equation*}
\frac{1}{9!}(q^6 + 2q^5 - 73q^4 + 344q^3 - 838q^2 + 1754q - 2030)(q^3-1)(q + 1)(q - 1)(q - 2)q^4.
\end{equation*}
\end{cor}

\begin{remark}
The computation of the $X^{q^2+q+1-9}Y^9$ coefficient of $W_{C_{2,3}^\perp}(X,Y)$ given in \cite[Theorem 3]{KaplanAGCT} has the additional restriction that the characteristic of $\F_q$ is not $2$ or $3$.  In Appendix \ref{AppendixA}, we compute $I_9(q)$ in a different way that is independent of the characteristic of $\F_q$, which shows that Corollary \ref{I9cor} holds in these additional cases.
\end{remark}

\subsection{Proofs of Lemmas \ref{lemmalowweightconics}, \ref{lemmalowweightconiccubic}, and \ref{lemmalowweightcubics}}

\begin{proof}[Proof of Lemma \ref{lemmalowweightconics}] 
Lemma \ref{supportconiclemma} implies that $C_{2,2}^\perp$ has no nonzero codewords of weight less than $4$ and that every codeword of weight $4$ is supported on a set of $4$ collinear points.  Proposition \ref{fdmprop} implies that there are $q-1$ weight $4$ codewords supported on any chosen set of $4$ collinear points.  

If $c_1,c_2 \in C_{2,2}^\perp$ and $\wt(c_1, c_2) = 4$, then at least one of $c_1, c_2$ must have weight $4$, and the other codeword is supported on the same set of $4$ collinear points.  There are $(q^2+q+1) \binom{q+1}{4}$ ways to choose the support of $(c_1, c_2)$ and $q^2-1$ choices of an ordered pair of codewords supported on these points where at least one is nonzero.

\end{proof}

\begin{proof}[Proof of Lemma \ref{lemmalowweightconiccubic}] 
Lemma \ref{supportconiclemma} implies that $C_{2,2}^\perp$ has no nonzero codewords of weight less than $4$. Lemma \ref{supportcubiclemma} implies that $C_{2,3}^\perp$ has no nonzero codewords of weight less than $5$.  Therefore, it is clear that $B_0^{[2]} = 1$ and $B_1^{[2]} = B_2^{[2]}  = B_3^{[2]} = 0$.  For each $i \in [4,6]$, the $B^{2,2}_i$ term given in the formula for $B_i^{[2]}$ accounts for pairs $(c_2, 0)$ where $c_2 \in C_{2,2}^\perp$ with $\wt(c_2) = i$.  For each $i \in [5,6]$, the $B^{2,3}_i$ term given in the formula for $B_i^{[2]}$ accounts for pairs $(0, c_3)$ where $c_3 \in C_{2,3}^\perp$ with $\wt(c_3) = i$.  Therefore, we need only consider the contribution from pairs $(c_2, c_3)$ where $c_2 \in C_{2,2}^\perp$ and $c_3 \in C_{2,3}^\perp$ are both nonzero.  In this case, if $\wt(c_2, c_3) \le 6$, then $\wt(c_3)$ is equal to $5$ or $6$. Lemma \ref{supportcubiclemma} implies that $c_3$ is supported on a set of $5$ or $6$ collinear points, so $\supp(c_2,c_3)$ must be a set of $5$ or $6$ collinear points.

Suppose that $\wt(c_2,c_3) = 5$. This implies that $\supp(c_2,c_3)$ is a set of $5$ collinear points.  Let $S$ be a set of $5$ collinear points.  Proposition \ref{fdmprop} implies that there are $q^2-1$ nonzero $c_2 \in C_{2,2}^\perp$ supported on $S$ and $q-1$ nonzero $c_3 \in C_{2,3}^\perp$ supported on $S$.  Every such pair has $\wt(c_2,c_3) = 5$.  Putting this together completes the calculation of $B_5^{[2]}$.

Now suppose that $\wt(c_2,c_3) = 6$.  This implies that $\supp(c_2,c_3)$ is a set of $6$ collinear points.  Let $S$ be a set of $6$ collinear points.  Proposition \ref{fdmprop} implies that there are $q^3-1$ nonzero $c_2 \in C_{2,2}^\perp$ supported on $S$ and $q^2-1$ nonzero $c_3 \in C_{2,3}^\perp$ supported on $S$.  We subtract the number of pairs where $\wt(c_2, c_3) = 5$.  We see that the contribution to $B_6^{[2]}$ from pairs where $c_2, c_3$ are both nonzero is
\[
(q^2+q+1)\binom{q+1}{6}\left((q^3-1)(q^2-1)- \binom{6}{5} (q^2-1)(q-1)\right).
\]
This completes the proof.

\end{proof}

\begin{proof}[Proof of Lemma \ref{lemmalowweightcubics}]
Lemma \ref{supportcubiclemma} implies that $C_{2,3}^\perp$ has no nonzero codewords of weight less than $5$ and that every nonzero codeword of weight at most $7$ is supported on a set of collinear points.  Therefore, it is clear that $B_0^{[2]} = 1$ and $B_1^{[2]} = B_2^{[2]}  = B_3^{[2]} = 0 = B_4^{[2]} = 0$.  For each $i \in [5,9]$, the $2B^{2,3}_i$ term given in the formula for $B_i^{[2]}$ accounts for pairs $(c, 0)$ and $(0,c)$ where $c \in C_{2,3}^\perp$ with $\wt(c) = i$.  Therefore, we need only consider the contribution from pairs $(c_1,c_2)$ with $c_1,c_2 \in C_{2,3}^\perp$ both nonzero and $\wt(c_1, c_2) \in [5,9]$.  Proposition \ref{gn_prop} implies that the number of such pairs where $\supp(c_1,c_2)$ is a set of $k$ collinear points is $g_3(k)$.

We need only count pairs with $\wt(c_1,c_2) = 8$ and $\wt(c_1,c_2) = 9$ where $\supp(c_1,c_2)$ is not a set of collinear points.  We first consider pairs with $\wt(c_1, c_2) = 8$.  We divide the count into cases based on $\supp(c_1, c_2)$.  By Lemma \ref{supportcubiclemma}, if $c\in C_{2,3}^\perp$ has weight at most $8$ then $\supp(c)$ is either a set of collinear points, $8$ points on a smooth conic, or $8$ points on two $\F_q$-rational lines, $4$ on each, and not including the intersection point.  

\begin{itemize}[wide, labelwidth=!, labelindent=0pt]  
\item If $\wt(c_1,c_2) = 8$ and the support of one of these codewords is a set of $8$ points on a smooth conic, then the other codeword is also supported on these same $8$ points.  Lemma \ref{VScubic8lemma} implies that given $8$ points on a smooth conic, there are $q-1$ nonzero $c \in C_{2,3}^\perp$ supported on them.  This case contributes $(q-1)^2 (q^5-q^2) \binom{q+1}{8}$ to $B_8^{[2]}$.

\item If $\wt(c_1,c_2) = 8$ and the support of one of these codewords is a set of $8$ points on of two $\F_q$-rational lines, $4$ on each, and not including the intersection point, then the other codeword is also supported on these $8$ points.  Lemma \ref{VScubic8lemma} implies that given such a collection of $8$ points, there are $q-1$ nonzero $c \in C_{2,3}^\perp$ supported on them.  This case contributes $(q-1)^2 \binom{q^2+q+1}{2} \binom{q}{4}^2$ to $B_8^{[2]}$.
\end{itemize}
\noindent This completes the analysis of $B_8^{[2]}$.

We now consider pairs with $\wt(c_1,c_2) = 9$.  By Lemma \ref{supportcubiclemma}, if $c\in C_{2,3}^\perp$ has $\wt(c) \le 9$ then $\supp(c)$ is either a set of collinear points, $8$ or $9$ points on a smooth conic, $8$ or $9$ points on two $\F_q$-rational lines with at least $4$ points on each, or $9$ points that are the intersection of two cubics that do not share a common component.  In this last case, every subset of $8$ of the points imposes independent conditions on cubics.

\begin{itemize}[wide, labelwidth=!, labelindent=0pt]  
\item If $\wt(c_1,c_2) = 9$ and the support of one of these codewords is contained in $9$ points of a smooth conic, then $\supp(c_1,c_2)$ is a set of $9$ points of a smooth conic.  Lemma \ref{VScubic9lemma} implies that given $9$ points of a smooth conic, there are $q^2-1$ nonzero $c \in C_{2,3}^\perp$ supported on them.  There are $(q^2-1)^2$ pairs of nonzero $c_1,c_2\in C_{2,3}^\perp$ with support contained in these $9$ points.  By Lemma \ref{VScubic8lemma},  $\binom{9}{8}(q-1)^2$ of these pairs have~$\wt(c_1,c_2) = 8$.  Therefore, this case contributes $(q^5-q^2)\binom{q+1}{9} \left((q^2-1)^2 - 9(q-1)^2\right)$ to $B_9^{[2]}$.

\item If $\wt(c_1,c_2) = 9$ and one of these codewords is supported on a collection of $9$ points that are the intersection of two cubics that do not share a common component, then the other codeword must also be supported on these $9$ points.  Lemma \ref{VScubic9lemma} implies that given $9$ such points, there are $q-1$ nonzero $c\in C_{2,3}^\perp$ supported on them.  Therefore, the contribution to $B_9^{[2]}$ from this case is $(q-1)^2 I_9(q)$, where  $I_9(q)$ is given in Corollary~\ref{I9cor}.

\item In every other case where $\wt(c_1,c_2) = 9,\ \supp(c_1) \cup \supp(c_2)$ is $9$ points contained in two $\F_q$-rational lines, with at least $4$ points on each.  We consider two cases.
\begin{enumerate}[wide, labelwidth=!, labelindent=0pt]  
\item  Suppose that $\supp(c_1,c_2)$ is a set of $9$ points on two $\F_q$-rational lines, with $5$ points on one line and $4$ points on the other line, not including the intersection point.  There are $(q^2+q+1)(q^2+q)\binom{q}{5} \binom{q}{4}$ ways to choose $9$ points of this type.  Lemma \ref{VScubic9lemma} implies that given $9$ such points, there are $q^2-1$ nonzero $c \in C_{2,3}^\perp$ supported on them.  Therefore, there are $(q^2-1)^2$ pairs of nonzero $c_1,c_2\in C_{2,3}^\perp$ with support contained in these $9$ points.  There are $(q-1)^2$ pairs where $c_1, c_2$ are supported on the $5$ collinear points. There are $\binom{5}{4}(q-1)^2$ pairs where $c_1, c_2$ are each supported on the same subset of $4$ of the $5$ collinear points together with the $4$ points of the other line.  Noting that 
\[
(q^2-1)^2 - 6(q-1)^2 = (q^4-8q^2+12q-5)
\] 
completes the analysis in this case.

\item Suppose that $\supp(c_1,c_2)$ is a set of $9$ points on two $\F_q$-rational lines consisting of the intersection point together with $4$ other points on each line.  There are $\binom{q^2+q+1}{2} \binom{q}{4}^2$ ways to choose $9$ points of this type.  Lemma \ref{VScubic9lemma} implies that given $9$ such points, there are $q^2-1$ nonzero $c \in C_{2,3}^\perp$ supported on them.  Therefore, there are $(q^2-1)^2$ pairs of nonzero $c_1,c_2\in C_{2,3}^\perp$ with support contained in these $9$ points.  These $9$ points contain two subsets of $5$ collinear points.  For each of these subsets, there are $(q-1)^2$ pairs with $c_1, c_2$ supported on these $5$ points.  There are also $(q-1)^2$ pairs where $c_1, c_2$ are both supported on the subset of $8$ of these $9$ points that does not include the intersection point.  Noting that
\[
(q^2-1)^2 - 3(q-1)^2 = (q^4-5q^2+6q-2)
\]
completes the analysis in this case.
\end{enumerate}

\end{itemize}

\end{proof}

\section{Intersections of Two Conics}\label{sec:twoconics}

The goal of this section is to prove a version of Theorem \ref{ThmIntCubics} for intersections of projective  conics.  The proof is much less intricate than the proof of Theorem \ref{ThmIntCubics}, but contains many of the same ideas.

\begin{thm}\label{ThmIntConics}
We have 
\begin{eqnarray*}
W^{[2]}_{C_{2,2}}(X,Y) & = & X^{q^2+q+1} + \frac{(q-1)q(q+1)^2(q^2+q+1)}{2} X^{2q+1} Y^{q^2-q} \\
& + & (q-1)^2 q^3 (q+1)(q^2+q+1) X^{q+2} Y^{q^2-1} \\
& + &  (q-1)(q+1)(q^2+q+1)(2q^3-q^2-q+1) X^{q+1} Y^{q^2} \\
& + & \frac{(q-1)^4 q^4 (q+1)^2 (q^2+q+1)}{24} X^4 Y^{q^2+q-3}\\
& + &  \frac{(q-1)^3 q^4 (q+1)^2(q^2+q+1)}{2} X^3 Y^{q^2+q-2} \\
& + & \frac{(q-1)^2 q^3 (q+1)^2 (q^2+q+1) (q^3-2q^2+7q-4)}{4} X^2 Y^{q^2+q-1} \\
& + & \frac{(q^3-q)(q^3-1)(2q^6+q^5-2q^4+5q^3+6q^2-6q+3)}{6} X Y^{q^2+q}\\
& + & \frac{(q-1)^3 q^4 (q+1)(q^2+q+1)(3q^2+1)}{8} Y^{q^2+q+1}.
\end{eqnarray*}
\end{thm}

Starting from the expression for $W_{C_{2,2}^\perp}(X,Y)$ given in Lemma \ref{lemmalowweightconics}, we prove Theorem \ref{ThmIntConics} in two steps.
\begin{enumerate}
\item We determine the contribution to $W^{[2]}_{C_{2,2}}(X,Y)$ from conics that share a common component.  This determines all but $5$ coefficients of $W^{[2]}_{C_{2,2}}(X,Y)$.
\item We apply Theorem \ref{MacWilliamsThm2} to $W^{[2]}_{C_{2,2}}(X,Y)$.  This gives $5$ linear equations that must be satisfied by the $5$ unknown coefficients.  A linear algebra calculation completes the proof.
\end{enumerate}

Recall from equation \eqref{EqSupp} that
\begin{eqnarray*}
W^{[2]}_{C_{2,2}}(X,Y) & = &  X^{q^2+q+1} + (q+1) \left(W_{C_{2,2}}(X,Y) - X^{q^2+q+1}\right) \\
& & + (q^2-1)(q^2-q) W_{C_{2,2}}^{(2)}(X,Y).
\end{eqnarray*}
Since we have already computed $W_{C_{2,2}}(X,Y)$ in Proposition \ref{WC22prop}, to determine $W^{[2]}_{C_{2,2}}(X,Y)$ we need only consider the contribution to $W^{[2]}_{C_{2,2}}(X,Y)$ from linearly independent pairs of elements in $C_{2,2}$, or equivalently, from distinct nonzero plane conics.

By B\'ezout's theorem, two conics that intersect in more than $4$ points share a common component.  If two distinct conics share a common component, that component must be an $\F_q$-rational line.  We have two possibilities: either one conic is a double line $L$ and the other is $L$ together with another $\F_q$-rational line, or both conics are pairs of $\F_q$-rational lines, with one of those lines in common.  In this last case we have two pairs of lines $\{L, L'\}$ and $\{L,L''\}$, and the number of $\F_q$-intersection points of these pairs depends on whether $L' \cap L''$ lies on $L$ or not.  Adding these cases together proves the following.
\begin{lemm}\label{LemmaC22}
Let $W_{C_{2,2}}^{[2],\text{com}}(X,Y)$ denote the contribution to $W_{C_{2,2}}^{[2]}(X,Y)$ from pairs of nonzero polynomials $(f,g)$ that define distinct conics that share a common component.  Then
\begin{eqnarray*}
W^{[2],\text{com}}_{C_{2,2}}(X,Y) & = &  (q^2+q+1) (q^2+q) q^2 (q-1)^2 X^{q+2} Y^{q^2-1} \\
& & +  (q^2-1)(q^2-q) (q^2+q+1)(q+1) X^{q+1} Y^{q^2},
\end{eqnarray*}
and
\begin{eqnarray*}
W^{[2]}_{C_{2,2}}(X,Y) & = & (q+1) W_{C_{2,2}}(X,Y) -qX^{q^2+q+1} + W^{[2],\text{com}}_{C_{2,2}}(X,Y) \\
&&  +  c_4 X^4 Y^{q^2+q-3}  + c_3 X^3 Y^{q^2+q-2} +c_2 X^2 Y^{q^2+q-1} \\
& & + c_1 X Y^{q^2+q} + c_0 Y^{q^2+q+1},
\end{eqnarray*}
for some values $c_0, c_1, c_2, c_3, c_4$.
\end{lemm}

We are now ready to prove Theorem \ref{ThmIntConics}.
\begin{proof}
Recall the expression for $W^{[2]}_{C_{2,2}^\perp}(X,Y)$ given in Lemma \ref{lemmalowweightconics}.  We apply Theorem \ref{MacWilliamsThm2} to the expression given in Lemma \ref{LemmaC22}.  Let $M$ be the $5 \times 5$ matrix with rows and columns labeled from $0$ to $4$ and $(i,j)$ entry equal to the $X^{q^2+q+1-i} Y^i$ coefficient of $(X+(q^2-1)Y)^j (X-Y)^{q^2+q+1-j}$. Let $\mathbf{c}$ be the column vector with entries $c_0,\ldots, c_4$.  Then $M \cdot \mathbf{c}$ must be equal to the column vector with $5$ entries, labeled from $0$ to $4$, whose $i$\textsuperscript{th} entry is the $X^{q^2+q+1-i} Y^i$ coefficient of
\begin{eqnarray*}
& & q^{12} W_{C_{2,2}^\perp}^{[2]}(X,Y) - \bigg( (q+1) W_{C_{2,2}}(X+(q^2-1)Y,X-Y)\\
& & -q(X+(q^2-1)Y)^{q^2+q+1}+ W_{C_{2,2}}^{[2],\text{com}}(X+(q^2-1)Y,X-Y)\bigg).
\end{eqnarray*}
A linear algebra computation in Sage shows that
\begin{eqnarray*}
c_0 & = & \frac{1}{8} (q + 1)  (q - 1)^3  q^4  (3q^2 + 1)  (q^2 + q + 1) 
 \\
c_1 &= & \frac{1}{6}  (q + 1)  (q - 1)^2  q^2  (q^2 + q + 1)  (2q^5 + q^4 - 2q^3 + 5q^2 + 6q - 6)  \\
c_2 &= & \frac{1}{4} (q - 1)^2  (q + 1)^2  q^3  (q^2 + q + 1)  (q^3 - 2q^2 + 7q- 4)  \\
c_3 &= & \frac{1}{2}  (q + 1)^2  (q - 1)^3  q^4  (q^2 + q + 1)  \\
c_4 & = & \frac{1}{24}  (q + 1)^2  (q - 1)^4  q^4  (q^2 + q + 1).
\end{eqnarray*}
\end{proof}

In an argument that depends on a calculation, it is natural to be a little skeptical.  As a check, we give a separate more geometric argument for the calculation of~$c_4$.

\begin{proof}[Second Proof for the Computation of $c_4$]
If two conics intersect in $4$ points and do not share a common component, then no $3$ of the intersection points are collinear. Let $S$ be a set of $4$ points in $\Proj^2(\F_q)$ with no $3$ collinear.  There are 
\[
\frac{(q^2+q+1)(q^2+q)q^2(q-1)^2}{4!}
\] 
ways to choose $4$ such points.  Let $\Gamma_{S,2}$ be as in the beginning of Section \ref{sec:supports}. Since $|S| = 4$, Lemma \ref{EGHlemma} implies that these points impose independent conditions on conics. Therefore, $|\Gamma_{S,2}| = q^2$.  There are $(q^2-1)^2$ pairs of nonzero quadratic polynomials vanishing on $S$.  For $(q^2-1)(q-1)$ of these pairs, the two polynomials define the same conic.  For every other pair, the two polynomials define conics that intersect in $S$ and do not share a common component.
\end{proof}

\subsection{Intersections of Affine Conics}

We compute $W^{[2]}_{C^A_{2,2}}(X,Y)$ following the strategy used for $W^{[2]}_{C_{2,2}}(X,Y)$.  We first count affine conics that share a common component.  Two distinct affine conics that share a common component must have a common $\F_q$-rational line.  Considering the possibilities for this pair of conics gives the analogue of Lemma \ref{LemmaC22}.

\begin{lemm}\label{LemmaC22A}
Suppose $q> 2$.  Let $W_{C^A_{2,2}}^{[2],\text{com}}(X,Y)$ denote the contribution to $W_{C^A_{2,2}}^{[2]}(X,Y)$ from pairs of nonzero polynomials $(f,g)$ that define distinct affine conics that share a common component.  Then
\begin{eqnarray*}
W^{[2],\text{com}}_{C^A_{2,2}}(X,Y) & = & (2q + 1)(q + 1)^2 (q - 1)^2 q^2 X^{q} Y^{q^2-q} \\
& & +  (q + 1)^2 (q - 1)^3 q^3 X^{q+1} Y^{q^2-q-1},
\end{eqnarray*}
and
\begin{eqnarray*}
W^{[2]}_{C^A_{2,2}}(X,Y) & = & (q+1) W_{C^A_{2,2}}(X,Y) -qX^{q^2} + W^{[2],\text{com}}_{C^A_{2,2}}(X,Y)\\
&  & + c_4 X^4 Y^{q^2-4} + c_3 X^3 Y^{q^2-3}  +c_2 X^2 Y^{q^2-2} \\
& & + c_1 X Y^{q^2-1} + c_0 Y^{q^2},
\end{eqnarray*}
for some values $c_0, c_1, c_2, c_3, c_4$.
\end{lemm}

An analysis of low-weight codewords of $(C_{2,2}^A)^\perp$ analogous to the one for $C_{2,2}^\perp$ shows that $(C_{2,2}^A)^\perp$ contains no nonzero codewords of weight less than $4$, and has ${(q-1)(q^2+q)\binom{q}{4}}$ codewords of weight $4$, exactly $q-1$ for each set of $4$ collinear points in the affine plane over $\F_q$.  It is then easy to show that if we write 
\[
W^{[2]}_{\left(C^A_{2,2}\right)^\perp}(X,Y) = \sum_{i=0}^{q^2} B^{[2]}_i X^{q^2 - i } Y^i,
\]
then $B^{[2]}_0 = 1,\ B^{[2]}_1 = B^{[2]}_2 = B^{[2]}_3 = 0$, and $B^{[2]}_4 = (q^2-1) (q^2+q) \binom{q}{4}$.

We apply Theorem \ref{MacWilliamsThm2} to the expression given in Lemma \ref{LemmaC22A}.  Let $M$ be the $5 \times 5$ matrix with rows and columns labeled from $0$ to $4$ and $(i,j)$ entry equal to the $X^{q^2-i} Y^i$ coefficient of $(X+(q^2-1)Y)^j (X-Y)^{q^2-j}$. Let $\mathbf{c}$ be the column vector with entries $c_0,\ldots, c_4$.  Then $M \cdot \mathbf{c}$ must be equal to the column vector with $5$ entries, labeled from $0$ to $4$, whose $i$\textsuperscript{th} entry is the $X^{q^2-i} Y^i$ coefficient of
\begin{eqnarray*}
& & q^{12} W^{[2]}_{\left(C^A_{2,2}\right)^\perp}(X,Y)- \bigg( (q+1) W_{C^A_{2,2}}(X+(q^2-1)Y,X-Y)\\
& & -q(X+(q^2-1)Y)^{q^2}+ W_{C^A_{2,2}}^{[2],\text{com}}(X+(q^2-1)Y,X-Y)\bigg).
\end{eqnarray*}
A linear algebra computation in Sage proves the following.
\begin{thm}\label{thmaffineconic}
Using the notation of Lemma \ref{LemmaC22A}, 
\begin{eqnarray*}
c_0 & = & \frac{3}{8} q(q+1)(q-1)^2 \left(q^8 + \frac{8}{9} q^7 + \frac{7}{3} q^6 - \frac{19}{9} q^5  + \frac{14}{3} q^4 + \frac{59}{9} q^3 - \frac{8}{3} q^2 + \frac{8}{3}\right)
 \\
c_1 &= & \frac{1}{3}  (q+1)(q-1)^2 q^3 \left(q^6 + 2 q^5 - \frac{5}{2} q^4 - \frac{29}{2} q^3 + \frac{15}{2} q^2 - \frac{27}{2} q + 3\right) \\
c_2 &= & \frac{1}{4}  (q+1)^2 (q-1)^3 q^4 \left(q^3 - 2q^2 +14 q -11 \right) \\
c_3 &= & \frac{2}{3}  \left(q - \frac{5}{4}\right) (q+1)^2 (q-1)^4 q^4 \\
c_4 & = & \frac{1}{24}  (q+1)^2(q-1)^4 q^4 (q^2-3q+3).
\end{eqnarray*}
\end{thm}

We will use this result in Section \ref{sec:large_int} when we count pairs of cubics that share an $\F_q$-rational line.

\section{Intersections of a Conic and a Cubic}\label{sec:con_cub}

In this section we explain how to use the strategy of Section \ref{sec:twoconics} to determine the number of pairs of $f\in \F_q[x,y,z]_2$ and $g\in \F_q[x,y,z]_3$ that have a given number of common zeros in $\Proj^2(\F_q)$.  Throughout the rest of this section suppose that $q>2$.

By B\'ezout's theorem, a conic and a cubic that intersect in more than $6$ points must share a common component.  We determine the contribution to $W^{[2]}_{C_{2,2},C_{2,3}}(X,Y)$ from pairs that share a common component.  This leaves $7$ unknown coefficients of $W^{[2]}_{C_{2,2},C_{2,3}}(X,Y)$.
\begin{lemm}\label{LemmaC2C3}
Let $W^{[2],\text{com}}_{C_{2,2},C_{2,3}}(X,Y)$ denote the contribution to $W^{[2]}_{C_{2,2},C_{2,3}}(X,Y)$ from pairs of nonzero polynomials $(f,g)$ defining a nonzero conic and a nonzero cubic that share a common component.  Then
\begin{eqnarray*}
W^{[2],\text{com}}_{C_{2,2},C_{2,3}}(X,Y) & = & \frac{q^2-q}{2} (q-1)^2 (q^2+q+1)^2 X Y^{q^2+q} \\
& & + \frac{1}{2}q(q+1)(q^3-1)^2 X^{2q+1} Y^{q^2-q} \\
& & +  \frac{1}{2} (q^2 + q + 1)(q + 1)(q - 1)^3 q^5 X^{q+3} Y^{q^2-2} \\
& & +  2(q^2 + q + 1)(q + 1)(q - 1)^2 q^5 X^{q+2} Y^{q^2-1} \\
& & +  \frac{1}{2} (q^2+q+1)(q-1)^2(q^7+5q^5+4q^4+2q^3+2q+2) X^{q+1} Y^{q^2},
\end{eqnarray*}
and
\begin{eqnarray*}
W^{[2]}_{C_{2,2},C_{2,3}}(X,Y)& = & W_{C_{2,2}}(X,Y) + W_{C_{2,3}}(X,Y)  - X^{q^2+q+1} +W^{[2],\text{com}}_{C_{2,2},C_{2,3}}(X,Y)\\
& & +  c_6 X^6 Y^{q^2+q-5} + c_5 X^5 Y^{q^2+q-4} + c_4 X^4 Y^{q^2+q-3} + c_3 X^3 Y^{q^2+q-2} \\
& & +   c_2 X^2 Y^{q^2+q-1} + c_1 X Y^{q^2+q} + c_0 Y^{q^2+q+1},
\end{eqnarray*}
for some values $c_0, c_1, c_2, c_3, c_4, c_5, c_6$.
\end{lemm}

\begin{proof}
If a conic and a cubic share a common component, that component must either be: a pair of Galois-conjugate lines defined over $\F_{q^2}$ but not over $\F_q$, a smooth conic, a pair of distinct $\F_q$-rational lines, or a single $\F_q$-rational line.  
\begin{itemize}
\item If a conic and a cubic share a pair of Galois-conjugate lines defined over $\F_{q^2}$ but not over $\F_q$, then the conic is this pair of lines, and the cubic is this pair of lines together with an additional $\F_q$-rational line.  The only $\F_q$-rational point in the intersection of such a conic and cubic, is the single rational point of the pair of Galois-conjugate lines.  There are $(q^2+q+1)(q^2-q)/2$ pairs of Galois-conjugate lines, $q^2+q+1$ choices for the additional line of the cubic. We include a factor of $(q-1)^2$ to account for scaling.

\item There are $q^5-q^2$ smooth conics defined over $\F_q$. The number of nonzero cubic polynomials vanishing on a given smooth conic is $q^3-1$.  

\item There are $\binom{q^2+q+1}{2}$ pairs of distinct $\F_q$-rational lines.  The number of nonzero cubic polynomials vanishing on a given pair of lines is $q^3-1$.  

\item There are $(q-1)(q^2+q+1)$ quadratic polynomials that define a double line.  There are $q^6-1$ nonzero cubic polynomials vanishing on a given line.
\end{itemize}

In every other case in which a conic and a cubic share a common component, the conic factors as a product of a pair of distinct $\F_q$-rational lines, and the cubic contains one, but not both, of those lines.  The number of $\F_q$-points of the intersection depends on how the cubic intersects the second line.  

We focus on the particular case where the conic is defined by the polynomial $xz$, and the cubic contains the line $x = 0$ but not $z=0$.  At the end of the calculation we multiply by $(q^2+q+1)(q^2+q)$ to account for the choice of a pair of lines and the choice of the shared line.

A cubic polynomial that vanishes on the line $x=0$ is of the form
\[
f_3(x,y,z) = x(a_0 x^2 + a_1 xy + a_2 xz + a_3 y^2 + a_4 yz + a_5 z^2).
\]
Since the cubic does not contain the line $z=0$ at least one of $a_0, a_1, a_3$ is nonzero.  We want to count the number of zeros of this polynomial on the line $z=0$ away from the point $[0:1:0]$.  Therefore, we need only count the number of zeros of $a_0 + a_1 y + a_3 y^2$ on the affine line with coordinate $y$.  Note that this does not depend on $a_2, a_4,$ or $a_5$, which leads to a factor of $q^3$ at the end of the count.

There are $q^3-1$ choices of $(a_0,a_1, a_3)$:
\begin{enumerate}
\item There are $(q-1) \binom{q}{2}$ choices that give a quadratic polynomial with distinct $\F_q$-rational roots.
\item There are $(q-1) q$ choices that give a quadratic polynomial with a double root.
\item If $a_3 = 0$ and $a_1 \neq 0$ we get a linear polynomial with a single $\F_q$-rational root.
\item If $a_3 = a_1 = 0$ and $a_0 \neq 0$ we get a constant polynomial with no $\F_q$-rational roots.
\item There are $(q-1) \frac{q^2-q}{2}$ choices that give an irreducible quadratic polynomial.
\end{enumerate}
Combining these observations completes the proof. For example note that 
\begin{eqnarray*}
& & (q-1)(q^5-q^2)(q^3-1)+(q-1)(q^2+q+1)(q^6-1) \\
& &+(q-1)(q^2+q+1)(q^2+q)q^3\left((q-1)\frac{q^2-q}{2} + (q-1)\right) \\ 
& & = \frac{1}{2}(q^2+q+1)(q-1)^2 \left(q^7+5q^5+4q^4+2q^3+2q+2\right).
\end{eqnarray*}

\end{proof}

We are now ready to prove the following result.
\begin{thm}
Using the notation of Lemma \ref{LemmaC2C3}, 
\begin{eqnarray*}
c_0 & = & \frac{53}{144} (q-1)^3 q^6 (q^2+q+1) \left(q^5 + q^4 + \frac{9}{53} q^3 + \frac{27}{53} q^2 + \frac{58}{53} q - \frac{32}{53}\right)
 \\
c_1 &= & \frac{11}{30}  q^4 (q-1)^2 (q^2+q+1) (q+1) \cdot\\ 
& &  \left(q^7 + \frac{1}{44} q^6 + \frac{5}{11} q^5 + \frac{20}{11} q^4 - \frac{31}{11} q^3 +\frac{159}{44} q^2 + \frac{15}{11}q - \frac{30}{11}\right) \\
c_2 &= & \frac{3}{16} (q-1)^2 (q+1)^2 q^5 (q^2+q+1) \left(q^5 - \frac{2}{9} q^4 + \frac{35}{9}q^3 - \frac{70}{9} q^2 + \frac{160}{9} q - \frac{32}{3} \right) \\
c_3 &= & \frac{1}{18}  (q+1)(q-1)^3 q^5 (q^2+q+1) \left(q^5+\frac{9}{2} q^4 + \frac{3}{2} q^3 + \frac{39}{2} q^2 + \frac{79}{2} q - 9\right) \\
c_4 & = & \frac{1}{48}  (q+1)(q-1)^3 q^6 (q^2+q+1) \left(q^4 + 13 q^2 + 26q - 48\right)\\
c_5 &= & \frac{1}{24}  (q+1)(q-1)^4 q^6 (q^2+q+1) (q^2+2q-5) \\
c_6 & = & \frac{1}{720}  (q-2)(q+1)(q-1)^4q^6 (q^2+q+1)(q^2+3q-8).
\end{eqnarray*}
\end{thm}

\begin{proof}
We apply Theorem \ref{MacWilliamsThm2} to the expression given in Lemma \ref{LemmaC2C3}.  Let $M$ be the $7 \times 7$ matrix with rows and columns labeled from $0$ to $6$ and $(i,j)$ entry equal to the $X^{q^2+q+1-i} Y^i$ coefficient of $(X+(q^2-1)Y)^j (X-Y)^{q^2+q+1-j}$. Let $\mathbf{c}$ be the column vector with entries $c_0,\ldots, c_6$.  Then $M \cdot \mathbf{c}$ must be equal to the column vector with $7$ entries, labeled from $0$ to $6$, whose $i$\textsuperscript{th} entry is the $X^{q^2+q+1-i} Y^i$ coefficient of
\begin{eqnarray*}
& & q^{16} W^{[2]}_{C_{2,2}^\perp, C_{2,3}^\perp}(X,Y)- \bigg( W_{C_{2,2}}(X+(q^2-1)Y,X-Y) + W_{C_{2,3}}(X+(q^2-1)Y,X-Y) \\
& & -(X+(q^2-1)Y)^{q^2+q+1}+ W_{C_{2,2},C_{2,3}}^{[2],\text{com}}(X+(q^2-1)Y,X-Y)\bigg).
\end{eqnarray*}
The $X^{q^2+q+1-i}Y^i$ coefficient of $W_{C_{2,3}}(X+(q-1)Y,X-Y)$ is a polynomial in $q$ for each $i \in [0,9]$ \cite[Section 4]{KaplanAGCT}.  For the same reason, for each $i \in [0,9]$ the $X^{q^2+q+1-i}Y^i$ coefficient of $W_{C_{2,3}}(X+(q^2-1)Y,X-Y)$ is also a polynomial in $q$.  

A linear algebra calculation in Sage completes the proof.

\end{proof}

Just as we did following the proof of Theorem \ref{ThmIntConics}, we include a separate more geometric argument for the computation of one of the weight enumerator coefficients.

\begin{proof}[Second Proof for the Computation of $c_6$]
If a conic and a cubic intersect in $6$ points and do not share a common component, then no $4$ of these points are collinear. By Lemma \ref{EGHlemma}, such a collection of points imposes independent conditions on cubics.  If $3$ of the points are collinear, then the conic must contain this line.  If a conic contains $6$ points with no $4$ on a line, and $3$ of those points are collinear, then the remaining $3$ points must also be collinear and none of the $6$ points is the intersection point of the two lines.  

Suppose $S$ is a collection of $6$ points in $\Proj^2(\F_q)$ with no $4$ collinear.  Since $S$ imposes independent conditions on cubics, there are $q^4-1$ nonzero cubic polynomials vanishing on $S$.  We divide the count into two cases.

\begin{enumerate}[wide, labelwidth=!, labelindent=0pt]  
\item Suppose $S$ is a set of $6$ points on a smooth conic.  There are $q-1$ nonzero quadratic polynomials vanishing on $S$, so there are $(q-1)(q^4-1)$ pairs of a nonzero quadratic polynomial and a nonzero cubic polynomial such that the corresponding conic and cubic contain $S$ in their intersection.  In $(q-1)(q^3-1)$ of these cases, the cubic consists of this conic together with an additional line.

\item Suppose $S$ is a set of $6$ points on two $\F_q$-rational lines, with $3$ points on each, not including the intersection point.  There are $q-1$ nonzero quadratic polynomials vanishing on $S$, so there are are $(q-1)(q^4-1)$ pairs of a nonzero quadratic polynomial and a nonzero cubic polynomial such that the corresponding conic and cubic contain $S$ in their intersection.  By B\'ezout's theorem, a cubic that contains a line and also contains $3$ collinear points that do not lie on that line, must contain the second line as well.  Therefore, there are $(q-1)(q^3-1)$ pairs such that the conic and cubic share a common component.
\end{enumerate}
Noting that
\begin{eqnarray*}
&(q^5-q^2) \binom{q+1}{6} (q-1)\left(q^4-1-(q^3-1)\right) +  \binom{q^2+q+1}{2}\binom{q}{3}^2 (q-1)\left(q^4-1-(q^3-1)\right) \\
& = \frac{1}{720}  (q-2)(q+1)(q-1)^4q^6 (q^2+q+1)(q^2+3q-8)
\end{eqnarray*}
completes the proof.
\end{proof}

\section{Cubic Curves that Share a Common Component}\label{sec:large_int}

The goal of this section is to prove analogues for a pair of cubics of Lemmas \ref{LemmaC22}, \ref{LemmaC22A}, and \ref{LemmaC2C3}.

\begin{prop}\label{LemmaC23}
Let $W_{C_{2,3}}^{[2],\text{com}}(X,Y)$ denote the contribution to $W_{C_{2,3}}^{[2]}(X,Y)$ from pairs of nonzero polynomials $(f,g)$ that define distinct cubic curves that share a common component.  Then
\begin{eqnarray*}
W_{C_{2,3}}^{[2],\text{com}}(X,Y) & = &  a_1 X Y^{q^2+q} + a_2 X^2 Y^{q^2+q-1} + a_{2q+1} X^{2q+1} Y^{q^2-q} \\
& &+ a_{2q+2} X^{2q+2} Y^{q^2-q-1}  + \sum_{i=q+1}^{q+5} a_{i} X^{i} Y^{q^2+q+1 - i},
\end{eqnarray*}
where 
\begin{eqnarray*}
a_1 & = & \frac{1}{2}(q-1)^2 (q^2+q+1) (q^2-q) (q+1)q \\
a_2 & = &  \frac{1}{2}(q-1)^2 (q^2+q+1) (q^2-q) (q^2+q)(q+1)q \\
a_{2q+1} & = & \frac{1}{2}(q^2+q+1)(2q+1)(q-1)^2q^2(q+1)^2 \\
a_{2q+2} & = & \frac{1}{2}(q^2+q+1)(q-1)^2(q^2+q)^2(q^2-q) \\
a_{q+1} & = & \cfrac{1}{24} (9q^8 + 8q^7 + 21q^6 - 19q^5 + 66q^4 + 59q^3 - 48q^2 +24)(q^2 + q + 1)(q + 1)(q - 1)^2 q \\
a_{q+2} &= & 	\frac{1}{6}(2q^5 + 2q^4 - 7q^3 + 42q^2 - 33q + 6)(q^2 + q + 1)(q + 1)^2(q - 1)^2q^3 \\
a_{q+3} &= & 	\frac{1}{4}(q^3 - 2q^2 + 14q - 11)(q^2 + q + 1)(q + 1)^2(q - 1)^3q^4 \\
a_{q+4} &= &  \frac{1}{6}(q^2 + q + 1)(4q - 5)(q + 1)^2(q - 1)^4 q^4 \\
a_{q+5} & = & \frac{1}{24}(q-1)^4q^4(q+1)^2(q^2+q+1)(q^2-3q+3).
\end{eqnarray*}
We have
\begin{eqnarray*}
W_{C_{2,3}}^{[2]}(X,Y)& = &  (q+1) W_{C_{2,3}}(X,Y) -qX^{q^2+q+1} + W_{C_{2,3}}^{[2],\text{com}}(X,Y) \\
& & +  \sum_{j=0}^9 c_j X^j Y^{q^2+q+1-j},
\end{eqnarray*}
for some values $c_0, c_1, \ldots, c_9$.
\end{prop}

\begin{proof}
Since we only consider pairs of polynomials $(f,g)$ that define distinct cubic curves, we may assume that $g$ is not a scalar multiple of $f$. B\'ezout's theorem implies that two distinct cubic curves that intersect in more than $9$ points share a common component.  There are several possibilities for two distinct cubics that share a common component:
\begin{itemize}[wide, labelwidth=!, labelindent=0pt]  
\item The two cubics share a common pair of Galois-conjugate lines defined over $\F_{q^2}$ but not over $\F_q$.  In this case, each cubic consists of this pair of conjugate lines together with an additional $\F_q$-rational line. 
\item The two cubics share a common smooth conic.  In this case, each cubic consists of this conic together with an $\F_q$-rational line.
\item The two cubics share a common pair of distinct $\F_q$-rational lines.  In this case, each cubic consists of these two lines together with another line, where we note that this third line might make one of the common lines into a double line.
\item The two cubics share a unique $\F_q$-rational line.  In this case there are many possibilities for how each cubic decomposes.  The important fact to note is that the number of intersection points of such a pair of cubics is in $[q+1,q+5]$.
\end{itemize}

The proof consists of case analysis.
\begin{enumerate}[wide, labelwidth=!, labelindent=0pt]  
\item We compute the contribution from polynomials that define cubic curves that share a pair of Galois-conjugate lines defined over $\F_{q^2}$ but not over $\F_q$.  The number of pairs of such lines is $(q^2+q+1) (q^2-q)/2$.  Each cubic contains an additional $\F_q$-rational line.  Since the cubics are distinct, these lines must be different.  Such a pair of cubics either intersects in exactly $1\ \F_q$-point, which is the case when the two rational lines contain the $\F_q$-point of the pair of Galois-conjugate lines, or $2\ \F_q$-points, which is the case when the two rational lines intersect at any of the $q^2+q$ other points of $\Proj^2(\F_q)$. There are $(q+1)q$ ordered pairs of distinct lines containing a chosen point.  We include a factor of $(q-1)^2$ to account for scaling each polynomial.  This completes the computation of $a_1$ and $a_2$.

\item We compute the contribution from polynomials that define cubic curves that share a pair of $\F_q$-rational lines.  There are $\binom{q^2+q+1}{2}$ such pairs.  Each cubic consists of these two common lines and another line, where that line may be one of the  common lines.  These two additional lines intersect in a unique $\F_q$-point.  There are $2q+1$ points on the two shared lines and $q^2-q$ points not on these lines.  There are $(q+1)q$ ordered pairs of distinct lines containing a chosen point.  We include a factor of $(q-1)^2$ to account for scaling each polynomial.  This completes the computation of $a_{2q+1}$ and~$a_{2q+2}$.

\item  We compute the contribution from polynomials that define cubic curves that share a smooth conic.  There are $q^5-q^2$ choices for the conic.  Each cubic contains this common conic and an additional line.  These lines intersect in a unique $\F_q$-point.  There are $q+1$ points on the conic and $q^2$ points not on the conic.  There are $(q+1)q$ ordered pairs of distinct lines containing a chosen point.  We include a factor of $(q-1)^2$ to account for scaling each polynomial. So the contribution in this case is
\[
(q^5-q^2) (q+1)q(q-1)^2 \left( (q+1) X^{q+1} Y^{q^2} + q^2 X^{q+2} Y^{q^2-1} \right).
\]

\item We compute the contribution from polynomials that define cubic curves that share a unique $\F_q$-rational line.  There are $q^2+q+1$ choices for the line.  Each cubic consists of this common line together with a plane conic that may be reducible. We need to count the number of $\F_q$-rational intersection points of these conics not on the common line.  This is equivalent to counting the number of $\F_q$-rational intersection points of the pair of affine conics we get from taking the common line as the `line at infinity' in $\Proj^2$.  That is, the contribution from this case is 
\[
(q^2+q+1) \left(\sum_{i=0}^4 c_i X^{q+1+i} Y^{q^2-i}\right),
\]
where the values of $c_0,c_1, c_2, c_3$, and $c_4$ are given in Theorem \ref{thmaffineconic}.  We do not need to include a factor of $(q-1)^2$ to account for scaling each polynomial because this scaling is already included in the result of Theorem \ref{thmaffineconic}.
\end{enumerate}
Combining these cases completes the proof of the theorem.

\end{proof}
As an additional check, we explain how to compute the contribution from this final case in a different way.  Two cubics that share a unique $\F_q$-rational line consist of that common line together with a pair of projective conics.  The number of $\F_q$-points in the intersection of these conics is in $[0,4]$. If this pair of cubics intersect in $q+5$ rational points, then the conics must intersect in $4\ \F_q$-points, and the shared line cannot contain any of them.  We first choose the pair of conics and then choose the line.  

The number of pairs of projective conics intersecting in $4\ \F_q$-points is given as $c_4$ at the end of the proof of Theorem \ref{ThmIntConics}.  If two conics that do not share a component intersect in $3$ or $4\ \F_q$-points, then no $3$ of these intersection points are collinear.  Therefore, the number of lines not containing any of these $4$ intersection points is $q^2+q+1 - (4(q+1)- \binom{4}{2})$.  This gives another proof that 
\[
a_{q+5} = \frac{1}{24} (q^2+q+1) (q+1)^2(q-1)^4 q^4 (q^2-3q+3).
\]
We can verify the other computations from this final case using this method but do not give the details here.

\section{Intersections of Two Cubics: The Proof of Theorem \ref{ThmIntCubics}}\label{sec:pf_thm1}

We are now ready to prove Theorem \ref{ThmIntCubics}.

\begin{proof}
We recall the expression for the low-weight coefficients of $W_{C_{2,3}^\perp}^{[2]}(X,Y)$ given in Lemma \ref{lemmalowweightcubics}.  We apply Theorem \ref{MacWilliamsThm2} to the expression given in Proposition \ref{LemmaC23}.   Let $M$ be the $10 \times 10$ matrix with rows and columns labeled from $0$ to $9$, and $(i,j)$ entry equal to the $X^{q^2+q+1- i}Y^i$ coefficient of $(X+(q^2-1)Y)^j (X-Y)^{q^2+q+1-j}$. Let $\mathbf{c}$ be the column vector with entries $c_0,\ldots, c_9$.  Then $M \cdot \mathbf{c}$ must be equal to the column vector with $10$ entries, labeled from $0$ to $9$, whose $i$\textsuperscript{th} entry is the $X^{q^2+q+1-i} Y^i$ coefficient of 
\begin{eqnarray*}
& & q^{20} W_{C_{2,3}^\perp}^{[2]}(X,Y) - \bigg((X+(q^2-1)Y)^{q^2+q+1}\\
& &  + (q+1) \left(W_{C_{2,3}}(X+(q^2-1)Y,X-Y) -(X+(q^2-1)Y)^{q^2+q+1}\right)\\
& &  + W_{C_{2,3}}^{[2],\text{com}}(X+(q^2-1)Y,X-Y)\bigg).
\end{eqnarray*}
A linear algebra computation in Sage completes the proof.
\end{proof}

\section{Further Questions}\label{sec:further}

One could try to follow the strategy of the proof of Theorem \ref{ThmIntCubics} to study intersections of curves of higher degree.  
\begin{question}\label{Questionde}
Let $d$ and $e$ be positive integers and let $k \in [0,de]$ be an integer.  How many of the $\left(q^{\binom{d+2}{2}}- 1\right) \left(q^{\binom{e+2}{2}}-1\right)$ ordered pairs $(f,g)$ with $f \in \F_q[x,y,z]_d$ and $g \in \F_q[x,y,z]_e$, both nonzero, define curves that intersect in exactly $k\ \F_q$-points and do not share a common component?
\end{question}
In this paper, we answer this question when $(d,e) \in \{(2,2),(3,2),(3,3)\}$.  As mentioned in Remark \ref{Rem2}, in forthcoming work we study the case where $e = 2$ and $d$ is arbitrary, and get a polynomial formula for each $k$.   For any $(d,e)$, the main term as $q\to \infty$ is due to Entin \cite{Entin}; see Theorem \ref{ThmEntin} in the introduction.  We believe that the case $(d,e) = (3,3)$ is a kind of boundary for this problem.  In this final section, we discuss two approaches to Question \ref{Questionde}.

We consider analogues of the main steps of the proof of Theorem \ref{ThmIntCubics} for general $(d,e)$.
\begin{enumerate}[wide, labelwidth=!, labelindent=0pt]  
\item \emph{Determine the contribution to $W_{C_{2,d},C_{2,e}}^{[2]}(X,Y)$ from pairs $(f,g)$ that define a degree $d$ curve and a degree $e$ curve that do share a common component.}  

If we are able to complete this step, this leaves $de+1$ unknown coefficients of $W_{C_{2,d},C_{2,e}}^{[2]}(X,Y)$.

\item \emph{Determine the $de+1$ lowest-weight coefficients of $W_{C_{2,d}^\perp,C_{2,e}^\perp}^{[2]}(X,Y)$ by counting configurations of up to $de$ points that fail to impose independent conditions on curves of degree $d$, and on curves of degree $e$.}

If we are able to complete this step, we get $de+1$ linear conditions that must be satisfied by the $de+1$ unknown coefficients of $W_{C_{2,d},C_{2,e}}^{[2]}(X,Y)$.  We would hope that these conditions uniquely determine these coefficients.

\end{enumerate}
We briefly discuss two difficulties with this approach.

\subsection{Traces of Hecke Operators and Coefficients of $W_{C_{2,3}^\perp}(X,Y)$}

Suppose that $e = 3$ and $d$ is arbitrary.   We want to determine the $3d+1$ lowest-weight coefficients of $W_{C_{2,d}^\perp,C_{2,3}^\perp}^{[2]}(X,Y)$.  The contribution to the $X^{q^2+q+1-j} Y^j$ coefficient of this weight enumerator from pairs of the form $(0,c)$ where $c \in C_{2,3}^\perp$ is the $X^{q^2+q+1-j} Y^j$ coefficient of $W_{C_{2,3}^\perp}(X,Y)$.

Theorem \ref{C23LowWeightDual} states that for each $j \le 9$, the $X^{q^2+q+1-j} Y^j$ coefficient of $W_{C_{2,3}^\perp}(X,Y)$ is a polynomial in $q$.  Once $j \ge 10$, the corresponding coefficient is not a polynomial in $q$, but involves a contribution from the trace of the Hecke operator $T_q$ acting on spaces of cusp forms of weight at most $j+2$ for $\operatorname{SL}_2(\Z)$.  See \cite[Section 4]{KaplanAGCT} for a more detailed discussion of how these non-polynomial terms arise in these weight enumerator coefficients.  For a precise statement of the $X^{q^2+q-9}Y^{10}$ coefficient of $W_{C_{2,3}^\perp}(X,Y)$ when $q\ge 5$ is prime see \cite[Theorem 3]{KaplanAGCT}.   We conclude that when $d \ge 4$, one should not expect the first $3d+1$ lowest-weight coefficients of $W_{C_{2,d}^\perp,C_{2,3}^\perp}^{[2]}(X,Y)$ to be polynomial in $q$.  When both $d,e$ are greater than $3$ we expect similar, but even more complicated, behavior for these coefficients of these weight enumerators.

\subsection{Counting Points that Fail to Impose Independent Conditions on Curves of Degree $d$}
For the rest of this section, suppose that $k < \binom{d+2}{2}$.

A key to answering the kinds of questions about intersections of curves that we study in this paper is understanding the number of collections of $k$ points in $\Proj^2(\F_q)$ that fail to impose independent conditions on curves of degree $d$. This question has been extensively studied (see \cite{EGH,Harris, Miranda}), but we do not have a complete characterization for every degree $d$ and every $k$.  Let $X_{d,k}$ be the space $\{(C,p_1,\ldots, p_k)\}$ where $C$ is a plane curve of degree $d$ and $p_1,\ldots, p_k$ are distinct points of $C$.  Let $\Conf^k(\Proj^2)$ denote the configuration space of unordered $k$-tuples of distinct points in $\Proj^2$.  We have a projection map $\pi \colon X_{d,k} \rightarrow \Conf^k(\Proj^2)$.  Since a generic set of $k$ points in $\Proj^2$ will impose independent conditions on degree $d$ curves, the fiber over a point of $\Conf^k(\Proj^2)$ is generically isomorphic to $\Proj^{\binom{d+2}{d} - 1 -k}$.

Consider $\mathcal{F}_{k,d} \subset \Conf^k(\Proj^2)$, the subset consisting of collections of $k$ distinct points that fail to impose independent conditions on degree $d$ curves.  We can define $\mathcal{F}_{k,d}$ algebraically as follows.  Choose an affine representative for each point of $\Proj^2$ and consider the standard set of $\binom{d+2}{2}$ monomials, $x^i y^j z^k$ where $i+j+k = d$.  By evaluating this basis of monomials, an ordered collection of $k$ points in $\Proj^2$ gives $\binom{d+2}{2} \times k$ matrix $A$. The set of points imposes independent conditions on degree $d$ curves if and only if the rank of the matrix is equal to $k$.  Note that the rank does not depend on a choice of ordering of the points.  In this way, we see that $\mathcal{F}_{k,d}$ can be defined by the simultaneous vanishing of a set of $k \times k$ minors.

\begin{question}
\begin{enumerate}
\item What can we say about $\#\mathcal{F}_{k,d}(\F_q)$?  In cases where we cannot get an exact answer, what can we say about asymptotic behavior as $q \to \infty$?

\item Can one obtain information about the singular cohomology of $\mathcal{F}_{k,d}$?  Can one use this information along with tools from \'etale cohomology to deduce consequences for $\#\mathcal{F}_{k,d}(\F_q)$?

\end{enumerate}
\end{question}
A better understanding of the cohomology of $\mathcal{F}_{k,d}$ would likely shed light on the question of when the kinds of counting problems we discuss in this paper have polynomial or non-polynomial answers.

\section*{Acknowledgements}
The first author is supported by NSF Grant DMS-1802281.  Part of this work grew out of the PhD thesis of the first author.  He thanks Noam Elkies for his guidance and support throughout this project.  The authors thank Igor Dolgachev for helpful comments related to an earlier draft of this paper.  The authors thank Greg Kuperberg for computational assistance.  They also thank the referees for many helpful comments.

\appendix
\section{The number of collections of $9$ points in $\Proj^2(\F_q)$ that are the intersection of two cubics}\label{AppendixA}

The goal of this appendix is to prove the following result.
\begin{thm}\label{I9count}
Let $I_9(q)$ be the number of collections of $9$ points in $\Proj^2(\F_q)$ such that there exist two cubic curves intersecting at these $9$ points that do not share a common component.  Then,
\[
I_9(q) =\frac{1}{9!}(q^6 + 2q^5 - 73q^4 + 344q^3 - 838q^2 + 1754q - 2030)(q^2 + q + 1)(q + 1)(q - 1)^2(q - 2)q^4.
\]
\end{thm}
\begin{remark}
In Section \ref{sec:dual_coef} we analyzed codewords of $C_{2,3}^\perp$ of weight $9$.  Comparing these calculations to an earlier result of Kaplan proved this result, except that \cite[Theorem 3]{KaplanAGCT} contains the additional assumption that the characteristic of $\F_q$ is not $2$ or $3$.  In this appendix we give a proof of Theorem \ref{I9count} that works in all characteristics, and moreover, makes no reference to results from coding theory.
\end{remark}

Suppose $\{P_1,\ldots, P_9\}$ is a collection of $9$ points such that there exist two cubic curves containing these points that do not share a common component.  No $4$ of these $9$ points lie on a line, since a cubic  that contains $4$ collinear points contains that line.  No $7$ of these $9$ points lie on a conic, since a cubic that contains $7$ points on a conic contains that conic.

We next recall a version of what is often called the Cayley-Bacharch theorem, but is more accurately due to Chasles.
\begin{prop}\cite[Chapter V, Corollary 4.5]{Hartshorne}\label{Hart4.5}
Given $8$ distinct points $P_1,\ldots, P_8$ in the plane, no $4$ collinear, and no $7$ lying on a conic, there is a uniquely determined $P_9$ (possibly an infinitely near point) such that every cubic through $P_1,\ldots, P_8$ also passes through $P_9$.
\end{prop}

We focus on the case where $P_9$ is distinct from the first $8$ points.
\begin{prop}\label{prop_infnear}
Using the notation of Proposition \ref{Hart4.5}, $P_9$ is infinitely near to one of the first $8$ points if and only if one of the following cases holds after relabeling the points:
\begin{enumerate}
\item $P_1,\ldots, P_8$ lie on an absolutely irreducible singular cubic with one of the points being the singular point.

\item $P_1,\ldots, P_6$ lie on a smooth conic, and the line containing $P_7$ and $P_8$ also contains exactly one of the first $6$ points.

\item $P_1,P_2, P_3$ are collinear, and $P_4,P_5,P_6$ are collinear, where none of these $6$ points is the intersection point of these two lines, and the line containing $P_7$ and $P_8$ contains exactly one of the first $6$ points.
\end{enumerate}
\end{prop}

\noindent We thank Igor Dolgachev for helpful suggestions related to the argument below.
\begin{proof}
We first show that if $P_1\ldots, P_8$ are in one of the three configurations in the statement, then $P_9$ is infinitely near to one of the $P_i$.
\begin{enumerate}
\item Suppose $P_1,\ldots, P_8$ are as in the first case of the statement. Let $C$ be an absolutely irreducible cubic containing them with a singular point at one of the points.  Without loss of generality, suppose $C$ is singular at $P_1$.  Let $C'$ be any other cubic containing $P_1,\ldots, P_8$.  The cubics $C$ and $C'$ already intersect at multiplicity $9$, so it is not possible for $C$ and $C'$ to intersect at any points not in $\{P_1,\ldots, P_8\}$, or for $C$ and $C'$ to intersect at multiplicity greater than $1$ at any of $P_2,\ldots, P_8$.

\item Suppose $P_1,\ldots, P_8$ are in the second case of the statement. Let $C$  be the union of the conic through $P_1,\ldots, P_6$ and the line containing $P_7$ and $P_8$.  Let $C'$ be any other cubic containing $P_1,\ldots, P_8$.  The rest of the argument is identical to the previous case.

\item Suppose $P_1,\ldots, P_8$ are in the third case of the statement. Let $C$  be the  be the union of the lines through $P_1, P_2, P_3$, through $P_4,P_5,P_6$, and through $P_7,P_8$.  Let $C'$ be any other cubic containing $P_1,\ldots, P_8$.  The rest of the argument is identical to the previous case.
\end{enumerate}

Let $P_1,\ldots, P_8$ be distinct points in the plane, no $4$ collinear, and no $7$ lying on a conic.  By Proposition \ref{Hart4.5}, there is a uniquely determined $P_9$ such that every cubic through $P_1,\ldots, P_8$ also passes through $P_9$.  Suppose that $P_9$ is infinitely near to one of the first $8$ points.  Without loss of generality, suppose that $P_9$ is infinitely near to $P_1$.  We now prove that $P_1,\ldots, P_8$ are in one of the three configurations given in the statement.

The space of cubic polynomials containing $8$ points in the plane has dimension at least $2$.  Let $f,g \in \F_q[x,y,z]_3$ be linearly independent polynomials that define cubic curves vanishing at $P_1,\ldots, P_8$ with tangent direction at $P_1$ specified by $P_9$.  Then it is clear that there exists $\lambda \in \F_q^*$ such that $f-\lambda g$ is a nonzero cubic containing $P_1,\ldots, P_8$ that is singular at $P_1$.

Suppose that $S = \{P_1,\ldots, P_8\}$ is a set of distinct points in the plane with no $4$ collinear and no $7$ on a conic and there is a cubic containing these points with a singular point at $P_1$.  If this cubic is absolutely irreducible we are in the first case of the proposition.  

Suppose that the cubic is reducible.  Every reducible cubic defined over $\F_q$ either factors as a smooth conic defined over $\F_q$ together with an $\F_q$-rational line, or as a union of three lines.
\begin{enumerate}
\item If the cubic factors as a smooth conic $C$ and a line $L$, then $P_1 \in C \cap L$.  It is not possible to have $3$ of the remaining $7$ points of $S$ lie on $L$ and it is not possible to have $6$ of the remaining points of $S$ on $C$.  Therefore, $|C\cap S| = 6,\ |L \cap S| = 3$, and $P_1$ is the only point of $S$ in $C \cap L$.  After relabeling the points, we see that we are in the second case of the proposition.

\item Suppose the cubic factors as the union of $3$ lines, $L_1, L_2, L_3$.  Since the union of these $3$ points contains at least $8\ \F_q$-points and none of these lines contains more than $3\ \F_q$-points, we see that the $L_i$ are distinct and $\F_q$-rational.  The only singular points of the cubic are the intersection points of pairs of the lines.  Without loss of generality, $P_1$ is the intersection point of $L_1$ and $L_2$.  We see that each of $L_1$ and $L_2$ contains at most $2$ additional points from $S$, so $L_1 \cup L_2$ contains at most $5$ points of $S$. Since $L_3$ contains at most $3$ points from $S$, we see that $|(L_1\cup L_2) \cap S| = 5$ and $L_3$ contains exactly the $3$ reamaining points of $S$.  After relabeling the points, we see that we are in the third case of the proposition.
\end{enumerate}

\end{proof}
\begin{prop}\label{J8prop}
Let $J_8(q)$ be the number of collections of $8$ points in $\Proj^2(\F_q)$ such that:
\begin{enumerate}
\item No $4$ points are collinear and no $7$ points lie on a conic;
\item It is not the case that the $8$ points lie on an absolutely irreducible singular cubic with the singular point at one of the $8$;
\item It is not the case that $6$ of the points lie on a smooth conic and the line containing the remaining $2$ contains exactly one of the other $6$;
\item It is not the case that $6$ of the points lie on two $\F_q$-lines, with $3$ points on each and not containing the intersection point, and the line containing the remaining $2$ contains exactly one of the other $6$.
\end{enumerate}
Then $I_9(q) = J_8(q)/9$.
\end{prop}

\begin{proof}
This follows from Proposition \ref{Hart4.5}, Proposition \ref{prop_infnear}, and the observation that if $P_1,\ldots, P_9$ are $9$ points such that there are two cubics containing them that do not share a common component, then applying Proposition \ref{Hart4.5} to any subset of $8$ of the points gives this same set of $9$.
\end{proof}

\begin{lemm}\label{no4no7}
The number of collections of $8$ points in $\Proj^2(\F_q)$ with at least $7$ on a smooth conic is
\[
C_{\ge 7}(q) = (q^5-q^2)\left(\binom{q+1}{8} + \binom{q+1}{7} q^2\right).
\]
The number of collections of $8$ points in $\Proj^2(\F_q)$ with at least $4$ on a line is
\begin{eqnarray*}
L_{\ge 4}(q)  & = & (q^2+q+1)\bigg(\binom{q+1}{8} + \binom{q+1}{7} q^2 + \binom{q+1}{6} \binom{q^2}{2} + \binom{q+1}{5} \binom{q^2}{3}  \\
&   + & \binom{q+1}{4} \binom{q^2}{4}\bigg)  -  (q^2+q+1)(q^2+q) \binom{q}{4} \binom{q}{3} - \binom{q^2+q+1}{2} \binom{q}{4}^2\\
&   - &  \binom{q^2+q+1}{2} \binom{q}{3}^2(q^2-q).
\end{eqnarray*}
The number of collections of $8$ points in $\Proj^2(\F_q)$ with no $4$ points collinear and no $7$ points on a conic is:
\[
\binom{q^2+q+1}{8} - C_{\ge 7}(q) - L_{\ge 4}(q).
\]
\end{lemm}

\begin{proof}
The number of subsets of $8$ points in $\Proj^2(\F_q)$ is $\binom{q^2+q+1}{8}$.  There are $q^5-q^2$ smooth plane conics, so the computation of $C_{\ge 7}(q)$ is clear.  There are $q^2+q+1\ \F_q$-rational lines in $\Proj^2(\F_q)$.  The number of subsets of $8$ points in $\Proj^2(\F_q)$ with at least $5$ points on a line is therefore,
\[
(q^2+q+1)\left(\binom{q+1}{8} + \binom{q+1}{7} q^2 + \binom{q+1}{6} \binom{q^2}{2} + \binom{q+1}{5} \binom{q^2}{3}\right).
\]
There are $(q^2+q+1) \binom{q+1}{4} \binom{q^2}{4}$ ways to choose a subset $S$ of $8$ points in the plane with exactly $4$ points on a chosen line $L$ and $4$ other points.  However, it may be the case that this collection of $8$ points contains $5$ collinear points, or two subsets of $4$ collinear points.  The first case occurs if and only if the $4$ other points lie on a line $L'$ and $L \cap L'$ is one of the $4$ chosen points of $L$.  The number of ways to choose $8$ such points is the number of choices of $L$ times the number of choices of $L'$ times the number of ways to choose $3$ additional points of $L$ and $4$ additional points of $L'$.

The set $S$ contains exactly two subsets of $4$ collinear points if and only if one of the following conditions holds:
\begin{enumerate}
\item $S$ consists of $4$ points each on two lines $L, L'$, not containing the intersection point.

\item $S$ consists the intersection point of two lines $L, L'$ and $3$ additional points on each line, along with $1$ additional point not on $L \cup L'$.
\end{enumerate}
There are $\binom{q^2+q+1}{2} \binom{q}{4}^2$ ways to choose a subset of the first type.  Since two $\F_q$-rational lines contain $2q+1\ \F_q$-points, there are $q^2-q\ \F_q$-points not contained in two chosen lines.  Therefore, there are $\binom{q^2+q+1}{2} \binom{q}{3}^2 (q^2-q)$ subsets of the second type.

It is clear that if a collection of $8$ points has at least $4$ on a line, it cannot have $7$ on a smooth conic.

\end{proof}

\begin{lemm}\label{absirred}
The number of collections of $8$ points in $\Proj^2(\F_q)$ that lie on an absolutely irreducible singular cubic with one of the $8$ points as the singular point is 
\begin{eqnarray*}
& & (q^2+q+1)(q^3-q)q^2\binom{q}{7}+(q^2+q+1)(q^3-q)\frac{q^3-q^2}{2}\binom{q-1}{7}\\
& & +(q^2+q+1)(q^3-q)\frac{q^3-q^2}{2}\binom{q+1}{7}.
\end{eqnarray*}
\end{lemm}

\begin{proof}
There are $3$ isomorphism classes of absolutely irreducible singular cubic curves:
\begin{enumerate}
\item Cuspidal cubics, which have $q+1\ \F_q$-points;
\item Split nodal cubics, which have $q\ \F_q$-points;
\item Non-split nodal cubics, which have $q+2\ \F_q$-points.
\end{enumerate}
For the number of polynomials defining each type of cubic, see \cite[Lemma 1]{KaplanAGCT}. We divide each term given there by $q-1$ to account for the fact that we count plane curves rather than cubic polynomials.
\end{proof}

\begin{lemm}\label{6conic2line}
The number of collections of $6$ points on a smooth conic together with $2$ points not on that conic such that the line through those $2$ contains exactly one of the other $6$ points is 
\[
(q^5-q^2)\binom{q+1}{6}\left(6(q-5)\binom{q-1}{2}+6 \binom{q}{2}\right).
\]

\end{lemm}

\begin{proof}
Suppose $P_1,\ldots, P_6$ are points of a smooth conic.  There are $6(q-4)\ \F_q$-rational lines through exactly $1$ of these $6$ points.  There are $6$ lines tangent to the conic and $6(q-5)$ lines not tangent to the conic.  A tangent line contains $q$ points not on the conic and a non-tangent line contains $q-1$ points not on the conic.
\end{proof}

\begin{lemm}\label{332lemma}
The number of collections two collinear triples $P_1,P_2, P_3$ and $P_4,P_5,P_6$, where the intersection point is not included among these $6$ points, and two additional points $P_7,P_8$ such that the line containing them passes through exactly $1$ of the first $6$ points is 
\[
\binom{q^2+q+1}{2}\binom{q}{3}^2 3(q-3)\binom{q-1}{2}.
\]
\end{lemm}

\begin{proof}
The number of choices for two collinear triples $P_1,P_2, P_3$ and $P_4,P_5,P_6$, where the intersection point is not included among these $6$ points is 
\[
\binom{q^2+q+1}{2}\binom{q}{3}^2.
\]
Suppose $P_1, P_2, P_3$ lie on the line $L$ and $P_4,P_5,P_6$ lie on the line $L'$.  There are $q+1$ lines through each of the $6$ points.  Of these $q+1$, there is $1$ line that contains two additional points from these $6$ and $3$ lines that contain $1$ additional point from these $6$.  So, there are $q-3$ lines that do not contain any additional points from this set of $6$.  For each such line, there are $q-1$ points that are not contained in $L \cup L'$.  

We divide by $2$ to account for the fact that the choice of which line to label $L$ and which line to label $L'$ was arbitrary.
\end{proof}

Combining Proposition \ref{J8prop} with Lemmas \ref{no4no7}, \ref{absirred}, \ref{6conic2line}, and \ref{332lemma} completes the proof of Theorem \ref{I9count}.  As a consequence, we see that Theorem \ref{C23LowWeightDual} does hold for any $q>2$, even when the characteristic of $\F_q$ is $2$ or $3$.

\end{document}